\RequirePackage{amsthm}
\documentclass[sn-basic,Numbered]{sn-jnl}

\usepackage{graphicx}%
\usepackage{multirow}%
\usepackage{amsmath,amssymb,amsfonts}%
\usepackage{amsthm}%
\usepackage{mathrsfs}%
\usepackage[title]{appendix}%
\usepackage{xcolor}%
\usepackage{textcomp}%
\usepackage{manyfoot}%
\usepackage{booktabs}%
\usepackage{algorithm}%
\usepackage{algorithmicx}%
\usepackage{algpseudocode}%
\usepackage{listings}%

\usepackage[shortlabels]{enumitem}
\usepackage{mathtools} 

\usepackage{verbatim}
\theoremstyle{thmstyleone}%
\newtheorem{theorem}{Theorem}
\newtheorem{proposition}[theorem]{Proposition}%

\newtheorem{lemma}{Lemma}

\theoremstyle{thmstyletwo}%

\theoremstyle{thmstylethree}%
\newtheorem{definition}{Definition}%
\newtheorem{remark}{Remark}

\newtheorem{eexample}{Example}
\newtheorem*{assumptions}{Assumptions}

\algrenewcommand\algorithmicrequire{\textbf{Input:}}
\algrenewcommand\algorithmicensure{\textbf{Output:}}
\makeatletter
\newenvironment{breakablealgorithm}
{
    \begin{center}
        \refstepcounter{algorithm}
        \renewcommand{\caption}[1]
        {
            \addcontentsline{loa}{algorithm}{\protect\numberline{\thealgorithm}##1}
            \parbox{\textwidth}
            {
                \hrule height.8pt depth0pt \kern2pt
                {\raggedright\textbf{\fname@algorithm~\thealgorithm} ##1\par}
                \kern2pt\hrule\kern2pt
            }
        }
}
{
        \kern2pt\hrule\relax
    \end{center}
}
\makeatother

\newcommand{\makered}[1]{{\color{red}#1}}
\newcommand{\makeblue}[1]{{\color{blue}#1}}

\newcommand{\makeorange}[1]{{\color{black}#1}}

\newcommand{\C}{\mathbb{C}}
\newcommand{\N}{\mathbb{N}}
\newcommand{\R}{\mathbb{R}}
\newcommand{\T}{\mathbb{T}}

\newcommand{\Z}{\mathbb{Z}}
\newcommand{\Zo}{\mathbb{Z}\setminus\{0\}}

\newcommand{\Cnn}{\C^{n\times n}}

\newcommand{\bigO}{{\mathcal{O}}}
\newcommand{\dd}[1][x]{\,\operatorname{d}\!#1}

\newcommand{\ddt}{\frac{\dd[]}{\dd[t]}}
\newcommand{\ddtddt}{\frac{\dd[]^2}{\dd[t^2]}}
\newcommand{\ddtddtddt}{\frac{\dd[]^3}{\dd[t^3]}}

\newcommand{\spn}{\operatorname{span}}

\newcommand{\tr}{\operatorname{Tr}}

\newcommand{\mA}{\mathbf{A}}
\newcommand{\mB}{\mathbf{B}}
\newcommand{\mC}{\mathbf{C}}
\newcommand{\mD}{\mathbf{D}}
\newcommand{\mE}{\mathbf{E}}

\newcommand{\mG}{\mathbf{G}}

\newcommand{\mI}{\mathbf{I}}
\newcommand{\mJ}{\mathbf{J}}

\newcommand{\mP}{\mathbf{P}}
\newcommand{\mQ}{\mathbf{Q}}

\newcommand{\mR}{\mathbf{R}}

\newcommand{\mT}{\mathbf{T}}
\newcommand{\mU}{\mathbf{U}}
\newcommand{\mV}{\mathbf{V}}

\newcommand{\mX}{\mathbf{X}}
\newcommand{\mY}{\mathbf{Y}}

\newcommand{\msqrtE}{\mE^{1/2}}
\newcommand{\msqrtEinv}{(\msqrtE)^{-1}}

\newcommand{\tA}{\widetilde \mA}
\newcommand{\tB}{\widetilde \mB}

\newcommand{\tE}{\widetilde \mE}

\newcommand{\tJ}{\widetilde \mJ}

\newcommand{\tPi}{\widetilde\Pi}
\newcommand{\tR}{\widetilde \mR}

\newcommand{\tSigma}{\widetilde \Sigma}

\DeclareFontFamily{U}{mathx}{\hyphenchar\font45}
\DeclareFontShape{U}{mathx}{m}{n}{
      <5> <6> <7> <8> <9> <10>
      <10.95> <12> <14.4> <17.28> <20.74> <24.88>
      mathx10
      }{}
\DeclareSymbolFont{mathx}{U}{mathx}{m}{n}
\DeclareFontSubstitution{U}{mathx}{m}{n}
\DeclareMathAccent{\widecheck}{0}{mathx}{"71}
\DeclareMathAccent{\wideparen}{0}{mathx}{"75}

\newcommand{\hJ}{\widehat \mJ}
\newcommand{\hR}{\widehat \mR}

\newcommand{\mAH}{\mA_H}
\newcommand{\mAS}{\mA_S}

\newcommand{\mCH}{\mC_H}

\newcommand{\mHC}{m_{HC}}

\newcommand{\xk}{{w_k}}

\newcommand{\bbH}{\mathbb{H}}
\newcommand{\Hn}{\mathbb{H}_n} 

\newcommand{\ip}[2]{\langle {#1}, {#2} \rangle}
\newcommand{\ipBig}[2]{\Big\langle {#1}, {#2} \Big\rangle}
\newcommand{\norm}[1]{\| {#1} \|}
\newcommand{\normBig}[1]{\Big\| {#1} \Big\|}

\DeclareMathOperator{\const}{const}
\DeclareMathOperator{\diver}{div}
\DeclareMathOperator{\diag}{diag}

\DeclareMathOperator{\range}{range}
\DeclareMathOperator{\rank}{rank}
\DeclareMathOperator{\rot}{rot}

\newcommand{\hRzz}{\hR_{2,2}}
\newcommand{\VhRzz}{\big(\hR_{2,2}\big)^{1/2}}
\newcommand{\hJzz}{\hJ_{2,2}}
\newcommand{\hJkzz}{(\hJ_{k_1})_{2,2}}
\newcommand{\cB}{\mathcal{B}}
\newcommand{\cD}{\mathcal{D}}

\newcommand{\cH}{\mathcal{H}}

\newcommand{\cL}{\mathcal{L}}

\newcommand{\wcH}{\widetilde{\cH}}

\newcommand{\uI}{u_1}
\newcommand{\ukI}{\phi_{k,1}}
\newcommand{\uII}{u_2}
\newcommand{\ukII}{\phi_{k,2}}

\begin{document}


\title[Hypocoercivity in algebraically constrained partial differential equations]{Hypocoercivity in algebraically constrained partial differential equations with application to Oseen equations}


\author[1]{\fnm{Franz} \sur{Achleitner}}\email{franz.achleitner@tuwien.ac.at}
\equalcont{All authors contributed equally to this work.}

\author*[1]{\fnm{Anton} \sur{Arnold}}\email{anton.arnold@tuwien.ac.at}
\equalcont{All authors contributed equally to this work.}

\author[2]{\fnm{Volker} \sur{Mehrmann}}
\email{mehrmann@math.tu-berlin.de}
\equalcont{All authors contributed equally to this work.}

\affil*[1]{\orgdiv{Technische Universit\"at Wien (TU Wien)}, \orgname{Institute of Analysis and Scientific Computing}, \orgaddress{\street{Wiedner Hauptstr. 8-10}, \city{Wien}, \postcode{A-1040}, 
\country{Austria}}}

\affil[2]{\orgdiv{Technische Universit\"at Berlin (TU Berlin)}, \orgname{Institut f.~Mathematik MA 4-5}, \orgaddress{\street{Stra\ss{}e des 17.~Juni 136}, \city{Berlin}, \postcode{D-10623}, 
\country{Germany}}}



\abstract{
The long-time behavior of solutions to different versions of Oseen equations of fluid flow  on the 2D torus is analyzed using the concept of hypocoercivity. 
The considered models are isotropic Oseen equations where the viscosity acts uniformly in all directions and anisotropic Oseen-type equations with different viscosity  directions. 
The hypocoercivity index is determined (if it exists) and it is shown that similar to the finite dimensional case of ordinary differential equations and differential-algebraic equations it characterizes its decay behavior.
}




\maketitle

\maketitle

\tableofcontents

\section{Introduction}
\label{sec:intro}
This paper is concerned with the long-time behavior and hypocoercivity structure of (an)isotropic Oseen equations from fluid dynamics. The Oseen equations describe the flow of a viscous and incompressible fluid at low Reynolds numbers and they have the form
\begin{equation} \label{eq:Oseen1}
\begin{cases}
u_t &=-(b\cdot\nabla)u- \nabla p +\nu \Delta u\,, \quad t>0\,, \\
0&=-\diver u \,.
\end{cases}
\end{equation}
The Oseen equations arise when one linearizes the incompressible or nearly incompressible {\em Navier-Stokes equations} describing the flow of a Newtonian fluid,
\begin{align*}
{u}_t &= - (u\cdot\nabla) u - \nabla p+\nu \Delta u\,, \quad t>0\,,\\
0&=-\diver u \,,
\end{align*}
around a prescribed vector field $b$, that is independent of space and time, see e.g. \cite{Shi73}.

As it also includes a (linear) convective term, it can be seen as an improvement of the flow description by the 
Stokes equations, see \cite[\S4.10]{Ba2000}, \cite[Chap.\ 2, \S11]{La1996}. 
The Oseen equations are a typical example of an operator differential-algebraic equation (DAEs) of the form
\begin{equation}\label{opDAE}
  \begin{bmatrix} \mI & 0 \\ 0 & 0  \end{bmatrix} 
  \frac{d}{dt} \begin{bmatrix} u \\ p \end{bmatrix}
  = -\begin{bmatrix} \mC & -\mD^T \\ \mD & 0  \end{bmatrix}  \begin{bmatrix} u \\ p \end{bmatrix}
\ ,\quad t>0\,,
\end{equation}
with the unbounded operators $\mC=b\cdot\nabla -\nu\Delta$ and $\mD=\diver$, see~\cite[p. 466]{EM13}.

DAEs of the form~\eqref{opDAE} also arise when the Oseen system is semi-discretized in space e.g. via a finite element discretization. This constitutes what is often called a vertical method of lines approach, see e.g., \cite{Jo16,Lay08,Ran00}.


In all described cases the equations have to be supplemented by suitable initial and boundary conditions. While in applications, \makeorange{see e.g. \cite{La1996,BaJi08,Gal11}}, the Oseen equation is typically considered on subsets of $\R^d$ or on unbounded exterior domains, to keep the presentations and the technicalities of our analysis simple, we analyze its long-time behavior here on the torus $\T^2 := (0,2\pi)^2$, \makeorange{similar to~\cite{FMRT01,T95,T97}}.

We perform the analysis using the concept of \emph{hypocoercivity} which was introduced 
in \cite{Vi09} in the study of unconstrained evolution equations (mostly partial differential equations) of the form $\ddt{x}=-\mC x$ on some Hilbert space~$\cH$, where the (possibly unbounded) operator~$-\mC$ generates a uniformly exponentially stable $C_0$-semigroup $(e^{-\mC t })_{t\geq 0}$.
More precisely, for hypocoercive operators~$\mC$ there exist constants $\mu>0$ and $c\ge1$, such that
\begin{equation}\label{exp-decay}
  \|e^{-\mC t} x_0\|_{\widetilde \cH}
  \le c\,e^{-\mu t} \|x_0\|_{\widetilde \cH}
  \qquad
  \mbox{\rm for all}\,x_0\in \widetilde \cH\,,
  \qquad
  t \geq 0\,,
\end{equation}
where $\widetilde \cH$ is another Hilbert space, densely embedded in $(\ker \mC)^\perp\subset \cH$.
Often, the evolution equation~$\ddt{x}=-\mC x$ is also called hypocoercive.
For infinitesimal generators $-\mC$, an estimate~\eqref{exp-decay} with $c=1$ holds if (and only if) $\mC$ is coercive.

The long-time behavior of many systems exhibiting hypocoercivity has been studied frequently in recent years, including Fokker--Planck equations~\cite{AASt15,ArEr14,Vi09}, kinetic equations~\cite{DoMoSc09,DoMoSc15}, and reaction-transport equations of BGK-type~\cite{AAC16,AAC18}. In these works, in particular in \cite{AAC16,AAC18,ArEr14}, the issue was to determine  the sharp (i.e.\ maximal) exponential decay rate~$\mu$, while to determine at the same time the smallest multiplicative constant~$c\ge1$ is a rather recent topic, e.g. see~\cite{AAS19}. 
Also the short-time behavior of linear evolution equations and its link to the \emph{hypocoercivity index} was recently discussed for systems of ordinary differential equations in \cite{AAC22,AAM22b} and for Fokker--Planck equations \cite[Th.\ 3.6]{ASS20}.

In this paper we consider (unbounded) operators $\mC$ on a 
Hilbert space~$\cH$ such that the operator $\mC$ is \emph{accretive}, i.e. $\mC$ has a nonnegative self-adjoint part, and that $-\mC$ generates a \emph{contraction semigroup}, i.e., 
it satisfies $\|e^{-\mC t}\|_\cH \leq 1$ for all $t\geq0$. 

We characterize those accretive operators~$\mC$ which are hypocoercive, i.e., those for which $-\mC$ generates a uniformly exponentially stable $C_0$-semigroup. 
Furthermore, based on our characterization and following the Lyapunov theory on Hilbert spaces, see e.g. \cite{Da70,Gr90,HaWe97}, we will construct appropriate strict Lyapunov functionals.

The remainder of this paper is structured as follows: 
In \S\ref{sec:hypocoercivity} we review key notions from hypocoercivity for finite dimensional ODEs and DAEs, and then extend it to the Hilbert space case. 
In \S\ref{ssec:Stokes} we apply these techniques to analyze Oseen equations on the 2D torus where the viscosity acts uniformly in all directions.
In Section~\S\ref{ssec:Oseen:anisotropic} we then study two anisotropic Oseen-type equations on the 2D torus, where the viscosity is different in the different space directions and the drift is either constant or space dependent.

\subsection*{Notation}

The conjugate transpose (transpose) of a matrix $\mC$ is denoted by $\mC^H$ ($\mC^\top$).
The set of Hermitian matrices in $\Cnn$ is denoted by $\bbH_n$.
Positive definiteness (semi-definiteness) of $\mC\in\bbH_n$ is denoted by $\mC>0$ ($\mC\geq 0$).
The unique \makeorange{positive semi-definite} square root of a positive semi-definite Hermitian matrix~$\mR$ is denoted by $\mR^{1/2}$ and the real part of a complex number $z$ is denoted by $\Re(z)$.

For linear operators on Hilbert spaces we use the following notation. 
%
The set of (possibly unbounded) linear operators from a Hilbert space $\mathcal H$ (with inner product $\ip{\cdot}{\cdot}$) to itself is denoted by ${\mathcal L}({\mathcal H})$, the subset of bounded linear operators by ${\mathcal B}({\mathcal H})$.
For a linear operator $\mC\in{\cL}({\cH})$ with domain $\cD(\mC)$, $\mC^*$ denotes the adjoint operator of $\mC$.
A self-adjoint operator $\mC\in {\mathcal L}({\mathcal H})$ is called \emph{nonnegative} ($\mC\geq 0$) if $\ip{\mC x}{x}=\ip{x}{\mC^*x}\geq 0$ for all $x\in\cD(\mC)$;
$\mC$ is called \emph{positive} ($\mC>0$) if $\ip{\mC x}{x}> 0$ for all $x\in\cD(\mC)\setminus\{0\}$.

We consider spaces of square-summable doubly-infinite sequences taking values in $\C^n$, $n=1,2,3$ and denote them as $\ell^2(\Z;\C^n)$; if $n=1$ then we will use the short-hand notation $\ell^2(\Z)=\ell^2(\Z;\C)$.
$\dot H^1_{per}(\T^2)$ denotes the homogeneous Sobolev space $\{f\in H^1_{per}(\T^2)\,|\, \int_{\T^2} f \dd[x]=0\}$ of periodic functions on the 2D-torus $\T^2$, and $\dot H^{-1}_{per}(\T^2)$ is its dual space.

\section{Hypocoercivity for finite and infinite dimensional evolution equations}
\label{sec:hypocoercivity}
In this section we recall the concepts of \emph{hypocoercivity and the hypocoercivity index} for (constrained) evolution equations. Our discussion will be in three steps, starting with finite dimensional cases, then infinite dimensional operator evolution equations with bounded generators, and finally some cases of unbounded generators.
We begin with the class of ordinary differential equation (ODE) systems 
\begin{equation}\label{ODE:A}
 \dot{x}(t) = \mA x(t)\,, \qquad t\geq 0\,,
\end{equation}
with some function $x:[0,\infty)\to\C^n$ and a constant matrix~$\mA\in\Cnn$. The second class are differential-algebraic equation (DAE) systems 
\begin{equation}\label{DAE:EA}
\mE \dot{x}(t) = \mA x(t)\,, \qquad t\geq 0\,,
\end{equation}
for a pair~$(\mE,\mA)$ of constant matrices~$\mE,\mA\in\Cnn$ with $\mE=\mE^H$ positive semi-definite.

Note that if~$\mE\in\Hn$ is positive definite, then it has a positive definite matrix square root~$\msqrtE\in\Hn$, and by a change of variables~$y:= \msqrtE x$ and by scaling the equation by $\msqrtEinv$, the DAE~\eqref{DAE:EA} takes the form
\begin{equation}\label{ODE:A-tilde}
\dot{y}(t) =\tA y
\qquad\text{where }
\tA :=\msqrtEinv\mA \msqrtEinv \ .
\end{equation}
However, if $\mE$ is singular then the behavior of the two systems~\eqref{ODE:A}--\eqref{DAE:EA} is fundamentally different.

Writing a matrix $\mA\in\Cnn$ as the sum of its Hermitian part $\mAH =(\mA +\mA^H)/2$ and skew-Hermitian part $\mAS =(\mA -\mA^H)/2$, we have the following definition.
\begin{definition}[{Definition 4.1.1 of~\cite{Be18}}] \label{def:dissipative:finiteDim}
A matrix~$\mA\in\Cnn$  is called~\emph{dissipative} (resp. \emph{semi-dissipative}) if the Hermitian part~$\mAH$ is negative definite (resp. negative semi-definite). 
For a (semi-)dissipative matrix~$\mA\in\Cnn$, the associated ODE~\eqref{ODE:A} is called \emph{(semi-)dissipative Hamiltonian ODE}.
A DAE ~\eqref{DAE:EA} with (semi-)dissipative matrix $\mA\in\Cnn$ and positive semi-definite Hermitian matrix~$\mE\in\Cnn$ is called~ \emph{(semi-)dissipative Hamiltonian DAE}.
\end{definition}
The notion (semi-)dissipative Hamiltonian is motivated by the fact that if $\mAH=0$ and $\mE$ is the identity, then~\eqref{ODE:A} is a \emph{Hamiltonian system} with \emph{Hamiltonian} $\mathsf H=(x^H \mE x)/2$, see~\cite{MehMW18} and also
\cite[Remark 1]{AAM21} and~\cite[Theorem 3(E1)]{AAM21}.

In the following (to avoid too many indices), we often write semi-dissipative matrices~$\mA$ in the form~$\mA =\mJ -\mR$ with a skew-Hermitian matrix~$\mJ=\mAS$ and a positive semi-definite Hermitian matrix $\mR =-\mAH$.

For ODE systems, hypocoercivity and the hypocoercivity index are defined as follows:

\begin{definition}[{\cite{AAC22}}] \label{def:matrix:hypocoercive}
A matrix $\mC\in\Cnn$ is called~\emph{coercive} if its Hermitian part $\mCH$ is positive definite, and it is called~\emph{hypocoercive} if the spectrum of~$\mC$ lies in the \emph{open} right half plane.


Let $\mJ,\mR\in\Cnn$ satisfy $\mR=\mR^H\geq 0$ and $\mJ=-\mJ^H$.
The~\emph{hypocoercivity index (HC-index)~$m_{HC}$ of the matrix~$\mC=\mR-\mJ$} is defined as the smallest integer~$m\in\N_0$ (if it exists) such that
\begin{equation}\label{Tm:J-R}
 \sum_{j=0}^m \mJ^j \mR (\mJ^H)^j > 0 
\end{equation}
holds.
\end{definition}

\begin{remark}\label{rem:short-t-decay}
\begin{enumerate}
    \item[(1)] Clearly, \eqref{Tm:J-R} is equivalent to the condition 
    \begin{equation}\label{Tm:J-R-strict}
        \sum_{j=0}^m \mJ^j \mR (\mJ^H)^j \ge \kappa\mI
    \end{equation}
    for some $\kappa>0$, where $\mI$ denotes the identity matrix.
    This variant will be needed in the infinite dimensional case below. 
    Condition~\eqref{Tm:J-R-strict} was also used in \cite{AAC18}. 
    \item[(2)] By a non-trivial result, the HC-index characterizes the short-time decay of semi-dissipative Hamiltonian ODE systems $\dot x(t)=-\mC x(t)$: Its system matrix $\mC$ has a HC-index $m_{HC}\in\N_0$ if and only if
\begin{equation} \label{short-t-decay}
\|e^{-\mC t}\|_2 = 1-ct^{2m_{HC}+1} + \mathcal O(t^{2m_{HC}+2}) \quad \mbox{for }t\to 0^+,
\end{equation}
for some $c>0$. Here we used that $\|e^{-\mC t}\|_2$ is a real analytic function on some (small) time interval $[0,t_0)$, see Theorem 2.7(a) in \cite{AAC22}.

This classification can be extended to the DAE case by considering the HC-index of the dynamical part, see Proposition \ref{prop:DAE+HC-decay} in the Appendix 
\ref{app:DAEs:hypocoercivity}.
\end{enumerate}

\end{remark}

Condition \eqref{Tm:J-R} is equivalent to the well known \emph{Kalman rank condition}:
\begin{lemma}[{\cite[Proposition 1, Remark 4]{AAC18}}] \label{lem:Equivalence}
Let $\mJ,\mR\in\Cnn$ satisfy $\mR=\mR^H\geq 0$ and $\mJ=-\mJ^H$. 
Then the following conditions are equivalent:
\begin{enumerate}[(B1)]
\item \label{B:KRC}
There exists $m\in\N_0$ such that
\begin{equation}\label{condition:KalmanRank}
 \rank[{\mR},\mJ{\mR},\ldots,\mJ^m {\mR}]=n \,.
\end{equation}
\item \label{B:KRC'}
There exists $m\in\N_0$ such that
\begin{equation}\label{condition:KalmanRank'}
 \bigcap_{j=0}^m \ker\big(\mR^{1/2}\ \mJ^j\big) 
= \{0\} \,.
\end{equation}
\item \label{B:Tm}
There exists $m\in\N_0$ such that \eqref{Tm:J-R} holds.
%
\end{enumerate}
Moreover, the smallest possible~$m\in\N_0$ in~\ref{B:KRC}, \ref{B:KRC'}, and~\ref{B:Tm} coincide. 
\end{lemma}
A similar equivalence result (on a subspace describing the dynamics of the system) has also been shown for DAE systems in \cite{AAM21}, see  Appendix~\ref{app:DAEs:hypocoercivity} here.\\

The characterization of hypocoercive matrices is related to results in control theory (as discussed in~\cite[Remark 2]{AAM21}):
\begin{remark}[Connection to control theory:  finite dimensional setting] \label{rem:controlTheory}
Consider a state-space system
\begin{equation}\label{ODE:AxBu}
 \dot x(t) =\mA x +\mB u
\end{equation}
for constant matrices $\mA,\mB\in\Cnn$. 
A pair~$(\mA,\mB)$ of square matrices~$\mA,\mB\in\Cnn$ is called \emph{controllable} if the controllability matrix~$[\mB, \mA\mB, \mA^2 \mB,\ldots,\mA^{n-1} \mB]$ has full rank.
For a controllable pair~$(\mA,\mB)$, the smallest possible integer~$k$ such that the controllability (sub)matrix $[\mB, \mA\mB, \mA^2 \mB,\ldots,\mA^{k-1} \mB]$ has full rank, is called the \emph{controllability index}, see e.g.~\cite[\S 6.2.1]{Ch99}, \cite[\S 6.2.1]{Ka80}, or~\cite[\S 5.7]{Wo85}.
\begin{itemize}
 \item
For semi-dissipative matrices~$\mJ-\mR$, the HC-index of $\mR-\mJ$ is one less than the controllability index of~$(\mA,\mB)=(\mJ,\mR)$.
 \item
Moreover, for general input-output systems there is also the dual concept of \emph{observability}.
Conditions like~\eqref{condition:KalmanRank'} are most often formulated in the context of observability, but not 
frequently in the context of controllability.
%
 \item
Whereas the controllability/observability indices have a clear interpretation in the discrete-time setting, see e.g.~\cite[p.171]{Ch99}; their interpretation in the continuous-time setting is not so clear. 
In particular, we are not aware of a characterization which is comparable to~\eqref{short-t-decay} in Remark~\ref{rem:short-t-decay}.
\end{itemize}
\end{remark}

While the two conditions  \eqref{condition:KalmanRank} and \eqref{condition:KalmanRank'} 
are clearly equivalent in the finite dimensional ODE case, only the latter one may also be used for linear operators. 
In fact, the above results are easily extended to
evolution equations of the form 
\begin{equation}\label{C-eq}
  \ddt{x}=-\mC x\,,\quad t>0,
\end{equation}
with $\mC$ a bounded or unbounded operator 
on some infinite dimensional 
Hilbert space~$\cH$:

\begin{definition}[{\cite[\S V.3.10]{Kato}}]
A linear operator $\mC$ on a Hilbert space $\cH$, with domain  $\cD(\mC)$, is said to be \emph{accretive} if the numerical range of $\mC$ is a subset of the right half plane, that is, if $\Re\ip{\mC x}{x} \geq 0$ for all $x\in\cD(\mC)$.
In this case $-\mC$ is said to be \emph{dissipative}.
And $\mC$ is called \emph{coercive} if there exists $\gamma>0$ such that $\ip{\mC x}{x}\geq \gamma\|x\|^2$ for all $x\in\cD(\mC)$.  
\end{definition}
Note that in this definition we follow the convention in semigroup theory, see e.g.~\cite[Proposition 3.23]{EnNa00}; whereas to be consistent with Definition~\ref{def:dissipative:finiteDim} we would have to call such an operator \emph{semi-dissipative}.


The classical Lyapunov characterization of (uniformly) exponentially stable semigroups on finite-dimensional Hilbert spaces easily extends to infinite dimensional settings:
\begin{theorem}[{\cite[Theorem 4.1.3]{CuZw20}},\cite{Da70}] 
Suppose that $\mA$ is the infinitesimal generator of the $C_0$-semigroup $\mT(t)$ on the Hilbert space~$\cH$.
Then $\mT(t)$ is uniformly exponentially stable if and only if there exists a bounded positive operator $\mP\in\cB(\cH)$ such that
\begin{equation} \label{LMI:H}
 \ip{\mA x}{\mP x} +\ip{\mP x}{\mA x} =-\ip{x}{x} \qquad\text{for all } x\in\cD(\mA)\ .
\end{equation}
Equation~\eqref{LMI:H} is called \emph{Lyapunov equation}.

If $\mT(t)$ is uniformly exponentially stable, then the unique self-adjoint solution of~\eqref{LMI:H} is given by
\begin{equation}
 \mP x =\int_0^{\infty} \mT(s)^* \mT(s) x \dd[s]
\qquad\text{for } x\in\cH\ .
\end{equation}
\end{theorem}
Next we recall the definition of hypocoercive operators in~\cite{Vi09}, which generalizes Definition~\ref{def:matrix:hypocoercive}:
%
\begin{definition}[{\cite[\S I.3.2]{Vi09}}] \label{def:hypoinfdim}
Let~$\mC$ be a (possibly unbounded) operator on a Hilbert space~$\cH$ with kernel~$\ker\mC$.
Let~$\widetilde\cH$ be a Hilbert space, which is continuously and densely embedded in $(\ker\mC)^\perp$, endowed with a scalar product~$\ip{\cdot}{\cdot}_{\widetilde\cH}$ and norm $\|\cdot\|_{\widetilde\cH}$.
The operator $\mC$ is called \emph{hypocoercive} on $\widetilde\cH$ if $-\mC$ generates a uniformly exponentially stable $C_0$-semigroup $(e^{-\mC t })_{t\geq 0}$ on~$\widetilde\cH\hookrightarrow (\ker\mC)^\perp$, i.e.~\eqref{exp-decay} holds.
\end{definition}

Next we shall generalize the notion \emph{hypocoercivity index} to the infinite dimensional Hilbert space case. As in the finite dimensional case, we consider operators of the form $\mC=\mR-\mJ$ with $\mR$ self-adjoint and nonnegative, and $\mJ$ skew-adjoint. For technical reasons (related to the operator domain) we shall first discuss bounded operators $\mC$ and then some special situations of unbounded operators $\mC$.

%
\begin{definition}
\label{def:op:HC-index}
Consider bounded operators $\mR,\,\mJ\in\cB(\cH)$ on a Hilbert space~$\cH$ such that $\mR$ is self-adjoint and nonnegative, and $\mJ$ is skew-adjoint, i.e., $\mJ^*=-\mJ$.
The~\emph{hypocoercivity index (HC-index)~$m_{HC}$ of the (accretive) operator~$\mC:=\mR-\mJ\in{\mathcal B}({\mathcal H})$ } is defined as the smallest integer~$m\in\N_0$ (if it exists) such that 
\begin{equation}\label{Op:J-R}
  \sum_{j=0}^m \mJ^j \mR (\mJ^*)^j \ge\kappa\mI 
\end{equation}
for some $\kappa>0$. 
\end{definition}

Here we generalized the uniform condition \eqref{Tm:J-R-strict}, since the example $\mC=\mR=\diag(1/j;\,j\in\N)$ on $\cH=\ell^2(\N)$ would satisfy the non-uniform condition \eqref{Tm:J-R}. However, the corresponding semigroup satisfies $\|e^{-\mC t}\|=1,\,t\ge0$, hence $\mC$ is not hypocoercive. 

When defining next the hypocoercivity index for unbounded operators, we do not aim at the largest generality of equations \eqref{C-eq}, but rather we present a framework that covers the Oseen equations in \S\ref{ssec:Oseen:anisotropic} below. 
In our extension, the unbounded operator $\mC=\mR-\mJ$ is accretive 
with the following assumptions:
\begin{assumptions}\
\begin{enumerate}[label=(A\arabic*)]
  \item \label{prop:R}
The unbounded operator $\mR$ with dense domain $\mathcal D(\mR)\subset \cH$ is self-adjoint (and hence closed) and nonnegative in $\cH$. 
The operator $\mJ$ is bounded and skew-adjoint on $\cH$.
  \item \label{prop:J}
$\mJ$ satisfies $\mJ(\mathcal D(\mR))\subset \mathcal D(\mR)$.
\item \label{prop:R1/2}
For the self-adjoint (and hence closed) operator $\mR^{1/2}$ defined on $\mathcal D(\mR^{1/2})\supset\mathcal D(\mR)$ assume that $\mJ(\mathcal D(\mR^{1/2}))\subset \mathcal D(\mR^{1/2})$.  
\end{enumerate}
\end{assumptions}

Under these assumptions, standard arguments from semigroup theory show that $-\mR$ and hence also $-\mC=-\mR+\mJ$ are dissipative with domain~$\cD(\mR)$ and, hence, infinitesimal generators of $C_0$-semigroups of contractions on $\cH$, see  e.g. \S1.4 of \cite{Pazy}.
Moreover these semigroups are analytic, see~\cite[Theorem III.2.10]{EnNa00}:
Due to~\cite[Corollary II.4.7]{EnNa00}, the operator $-\mR$ generates a bounded analytic semigroup on~$\cH$, and $\mJ$ is $\mR$-bounded with $\mR$-bound $a_0=0$, see~\cite[Definition III.2.1]{EnNa00}.

For unbounded operators, the relation between the domains of the operator, its self-adjoint part and its skew-adjoint part can be subtle.
Therefore, different extensions of  Definition~\ref{def:op:HC-index} are reasonable.
For example, under the assumptions~\ref{prop:R}--\ref{prop:J}, each term of the sum~\eqref{Op:J-R} is well-defined on $\mathcal D(\mR)$. 
However, the following extension to $\cD(\mR^{1/2})$ is more convenient for the subsequent lemma:

\begin{definition}\label{def:unb-op:HC-index}
Let the (accretive) operator~$\mC=\mR-\mJ$ satisfy the Assumptions~\ref{prop:R}, \ref{prop:R1/2}. 
Then the~\emph{hypocoercivity index} (HC-index)~$m_{HC}$ of~$\mC$ is defined as the smallest integer~$m\in\N_0$ (if it exists) such that 
\begin{equation} \label{Op:J-R:unbdd}
 \sum_{j=0}^m \| \mR^{1/2}\ (\mJ^*)^j x\|^2  
\geq \kappa \|x\|^2 \qquad \mbox{for all } \,x\in\cD(\mR^{1/2})\ ,
\end{equation}
for some $\kappa>0$.
\end{definition}

In both of the above settings (Definitions~\ref{def:op:HC-index} and \ref{def:unb-op:HC-index}) we have the following infinite dimensional analog of Lemma \ref{lem:Equivalence}. 
%
\begin{lemma} \label{lem:Op-Equivalence}
Let the operators $\mR,\,\mJ\in\cL(\cH)$ on a Hilbert space~$\cH$ satisfy either the assumptions in Definition \ref{def:op:HC-index} (if $\mC=\mR-\mJ$ is bounded) or the assumptions in Definition \ref{def:unb-op:HC-index} (if $\mC$ is unbounded).
Then the following three conditions are equivalent:
\begin{enumerate}[(B1')]
\setcounter{enumi}{0}
\item \label{B:G*_surjective}
There exists $m\in\N_0$ such that
\[
\spn \Bigg(\bigcup_{j=0}^m \range\big(\mJ^j \mR^{1/2}\big)\Bigg) =\cH\ .
\]
\item \label{B:KRC'op}
There exists $m\in\N_0$ such that
\[
 \bigcap_{j=0}^m \ker\big(\mR^{1/2}\ \mJ^j\big) 
= \{0\} \,, \qquad
 \text{and }  
 \spn \Bigg(\bigcup_{j=0}^m \range\big(\mJ^j \mR^{1/2}\big)\Bigg)\ \text{ is closed.}
\] 

\item \label{B:Tmop}
There exists $m\in\N_0$ such that \eqref{Op:J-R:unbdd} holds for some $\kappa>0$. 
\end{enumerate}
 
Moreover, the smallest possible~$m\in\N_0$ coincides in all cases (if it exists).
\end{lemma}

\begin{proof}
To prove Lemma~\ref{lem:Op-Equivalence} we will make use of the (equivalent) characterizations of surjective operators in~\cite[Theorem 2.21]{Br11}.
To this end, we will introduce the following Hilbert spaces and operators. 

For $m\in\N_0$, the direct sum $\cH^{m+1}:=\bigoplus_{j=0}^m \cH =\{(y_0,\ldots,y_m)|\ y_j\in\cH,\ j=0,\ldots,m\}$ endowed with the inner product $\ip{x}{y}_{\cH^{m+1}}:=\sum_{j=0}^m \ip{x_j}{y_j}$; $x,y\in\cH^{m+1}$ is again a Hilbert space.
Define the linear operator
\[
 \mG:\
\cD(\mR^{1/2}) \subset\cH \to \cH^{m+1}\,, \qquad
x \mapsto \begin{bmatrix} \mR^{1/2}x \\ \mR^{1/2}\mJ^* x \\ \vdots \\ \mR^{1/2} (\mJ^*)^m x \end{bmatrix}\ .
\]
Since $\mR^{1/2}$ is closed and $\mJ$ is bounded, also the operators $\mR^{1/2}(\mJ^*)^j$ are closed. Hence the operator~$\mG$ is densely defined (in both cases of assumptions $\cD(\mG):=\cD(\mR^{1/2})$ is dense in $\cH$) and closed. 
Its adjoint reads 
\begin{align*}
 \mG^*:\ 
\cD(\mG^*) \subset\cH^{m+1} &\to \cH\,, \quad\mbox{with } \cD(\mG^*) \supset \big(\cD(\mR^{1/2})\big)^{m+1}\,, \\
y &\mapsto \sum_{j=0}^m \mJ^j \mR^{1/2} y_j\ .
\end{align*}

Next, we identify the operator $\mG$ and its adjoint $\mG^*$ in our statement with the ones in~\cite[Theorem 2.21]{Br11}:
First,
\begin{equation}\label{rangeG*}
 \range(\mG^*)
=\spn \Bigg(\bigcup_{j=0}^m \range\big(\mJ^j \mR^{1/2}\big)\Bigg) \stackrel{!}{=} \cH 
\end{equation}
gives the equivalence of the conditions \ref{B:G*_surjective} and (a) in \cite[Theorem 2.21]{Br11}, where ``$\stackrel{!}{=}$'' indicates the condition to be satisfied. 
Note that in \eqref{rangeG*} each $\range\big(\mJ^j \mR^{1/2}\big)$ has to be evaluated on the $j$-th component of $\cD(\mG^*)$, which may indeed be a proper superset of $\cD(\mR^{1/2})$. 

Second,
\[
 \ker(\mG)
= \bigcap_{j=0}^m \ker\big(\mR^{1/2} (\mJ^*)^j\big)
= \bigcap_{j=0}^m \ker\big(\mR^{1/2} \mJ^j\big)
\stackrel{!}{=} \{0\}\ .
\]
Moreover, due to the assumptions, $\range(\mG)$ is closed if and only if $\range(\mG^*)$ is closed.
This gives the equivalence of the conditions \ref{B:KRC'op} and (c) in \cite[Theorem 2.21]{Br11}. 

Third, 
\begin{multline*}
 \| \mG x\|_{\cH^{m+1}}^2 
= \sum_{j=0}^m \| (\mG x)_j\|^2
= \sum_{j=0}^m \| \mR^{1/2}\ (\mJ^*)^j x\|^2  
\stackrel{!}{\ge} \kappa \|x\|^2 
 \\
 \text{for all } x\in \cD(\mR^{1/2})=\cD(\mG) \ .
\end{multline*}
This gives the equivalence of the conditions \ref{B:Tmop} and (b) in \cite[Theorem 2.21]{Br11}. 

The three equivalences established in Theorem~\cite[Theorem 2.21]{Br11} thus imply the three equivalences of Lemma~\ref{lem:Op-Equivalence}.
\end{proof}

\begin{remark}[Connection to control theory: infinite dimensional setting]
Consider the control system~\eqref{ODE:AxBu} where, for simplicity, $\mA,\mB\in\cL(\cH)$ operate on the same Hilbert space~$\cH$.
In the infinite dimensional setting many more concepts of controllability exist, 
e.g.~depending on the operator domains~$\cD(\mA),\cD(\mB)\subset\cH$.
In \cite[Theorem 3.18]{CuPr78} and \cite[Theorem 6.2.27]{CuZw20} it is observed that for $\mA,\mB\in\cB(\cH)$, system~\eqref{ODE:AxBu} is \emph{exactly controllable} if and only if Condition \ref{B:G*_surjective} in Lemma \ref{lem:Op-Equivalence} holds.

\end{remark}

After recalling
 the basic hypocoercivity concepts, in the next two sections we apply the techniques for the analysis of different variants of the Oseen equations.
We consider the Oseen equations as a constrained PDE model on a torus. Since this allows for a modal decomposition, it reduces  to an infinite system of DAEs. 
Depending on the detailed shape of the drift and diffusion terms in the Oseen equation, these models exhibit a wide range of hypocoercivity phenomena: 
it may be coercive (see \S\ref{ssec:Stokes}), hypocoercive with index 1 (see \S\ref{ssec:nonconst-drift}), or not hypocoercive (see \S\ref{ssec:const-drift}).

\section{Isotropic Oseen equation on the 2D torus} \label{ssec:Stokes}


As first model  we consider the time-dependent, incompressible Oseen equation of fluid dynamics with \emph{isotropic} viscosity on the 2D torus $\T^2 := (0,2\pi)^2$,
\begin{subequations}\label{eq:Oseen}
\begin{align}
 u_t 
&=-(b\cdot\nabla)u -\nabla p +\nu \Delta u\,, 
&& t>0\,, \quad\text{on }\T^2\,, \label{eq:Oseen:1a} 
\\
 0 
&=-\diver u \,, 
&& t\geq 0\,, \label{eq:Oseen:1b} 
\end{align}
\end{subequations}
for the vector-valued velocity field~$u=u(x,t)$ and the scalar pressure~$p =p(x,t)$ in the space variable~$x\in\T^2$ and the time variable~$t\geq 0$.
The constant $\nu>0$ denotes the viscosity coefficient and $b\in\R^2$ is the constant drift velocity. 
Note that since in~\eqref{eq:Oseen} diffusion acts uniformly in all directions, we call the Oseen model~\eqref{eq:Oseen} \emph{isotropic}.

For \eqref{eq:Oseen} we assume periodic boundary conditions in both $u$ and $p$. 
Hence, this model actually could be simplified right away:
Taking the divergence of the first equation in~\eqref{eq:Oseen} yields
$\Delta p(\cdot,t) =0$ and hence~$p(\cdot,t)$ is constant in~$x$.
It also shows that the vector-valued transport-diffusion equation 
\begin{equation}\label{trans-diff}
  u_t=-(b\cdot\nabla)u +\nu \Delta u
\end{equation}
preserves the incompressibility if the initial condition satisfies $\diver u(0)=0$, which is assumed in the sequel. 
Since it is known that in this case the normal component of $u$ has a periodic extension \cite[\S IX.1.2+3]{DL3}, \cite[\S II.5]{FMRT01}, 
it follows that, for any initial condition 
\[
u(0)\in H_{per}(\diver 0, \T^2):=\{u\in (L^2(\T^2))^2\, |\,\diver u=0\},
\] 
equation \eqref{trans-diff} and hence \eqref{eq:Oseen} has a unique smooth solution for $t>0$, and its explicit Fourier representation can be obtained from the first equation of the Fourier expansion \eqref{eq:Stokes:Fourier} below.

Here, in order to pave the way for more general applications, we proceed differently and employ  negative hypocoercivity of matrix pencils in semi-dissipative Hamiltonian DAEs. So we ignore this possible simplification and rather follow our discussion from~\cite[\S3]{AAM21}. 
The following analysis is an extension of \S4.1 in \cite{AAM21}, which considered the Stokes equation.

Due to the periodic setting, we consider the Fourier expansion of~\eqref{eq:Oseen} with
\begin{equation}\label{Fourier}
u(x,t)
= \sum_{k\in\Z^2} \phi_k(t) e^{i k\cdot x} \,, \qquad
p(x,t)
= \sum_{k\in\Z^2} p_k(t) e^{i k\cdot x} \,.
\end{equation}
Since $u$ and $p$ are real valued, their Fourier coefficients~$\phi_k(t)\in\C^2$, $p_k(t)\in\C$, $k\in\Z^2$, obey the symmetry $\bar\phi_k=\phi_{-k}$, $\bar p_k=p_{-k}$ and they satisfy the decoupled evolution equations
\begin{equation}\label{eq:Stokes:Fourier}
\begin{cases}
\ddt \phi_k &= -i k p_k -i(b\cdot k)\phi_k -\nu |k|^2 \phi_k \ , \quad t>0 \ , \\
0 &=-i k\cdot \phi_k \ .
\end{cases}
\end{equation}
The mode~$k=0$ satisfies $\phi_0(t)=$ const (corresponding to momentum conservation) and~$p_0(t)=$ arbitrary.
To enforce unique solvability of~\eqref{eq:Oseen}, we require the pressure as $p_0(t)\equiv 0$. Hence, this flow is a \emph{transversal wave} \cite[p.73]{La1996}. 
For~$k\ne 0$ we write~\eqref{eq:Stokes:Fourier} as a system of decoupled DAEs.
%
\begin{equation} \label{DAE:Stokes}
 \mE \dot{w}_k (t) =\mA_k \xk \ , \quad t\geq 0\ ,
\end{equation}
for~$\xk :=[\ukI,\ukII,p_k]^\top \in\C^3$ with the matrices~$\mE:=\diag(1,1,0)$ and
\begin{equation} \label{DAE:Stokes:A}
\mA_k :=
 \begin{bmatrix}
  -i(b\cdot k)-\nu |k|^2 & 0 & -i k_1 \\
  0 & -i(b\cdot k)-\nu |k|^2 & -i k_2 \\
  -i k_1 & -i k_2 & 0
 \end{bmatrix} \ .
\end{equation}
The modal functions~$\xk(t)$, $k\in\Z^2$ correspond to the function~$x(t)$ in~\S\ref{sec:hypocoercivity} and \cite{AAM21}, since~$x=[x_1, x_2]^\top$ is used here for the spatial variable.
Following the notation from~\S\ref{sec:hypocoercivity}, we  decompose $\mA_k$ as $\mA_k =\mJ_k -\mR_k$ with $\mR_k :=\diag(\nu |k|^2,\nu |k|^2, 0)$ and
\begin{equation} \label{DAE:Stokes:J}
\mJ_k :=
 \begin{bmatrix}
  -i(b\cdot k) & 0 & -i k_1 \\
  0 & -i(b\cdot k) & -i k_2 \\
  -i k_1 & -i k_2 & 0
 \end{bmatrix} \ .
\end{equation}
In order to determine the hypocoercivity index of~\eqref{DAE:Stokes} we employ the unitary  transformation of~\eqref{DAE:Stokes} to  staircase form via the Algorithm in~\cite[Lemma 5]{AAM21}, which we recall as Lemma~\ref{lem:tSF} in Appendix~\ref{app:DAEs:hypocoercivity}.
Applying the staircase algorithm to the triple~$(\mE,\mJ_k,\mR_k)$ yields 
\begin{subequations}\label{checkJk}
\begin{align} 
\widecheck \mJ_k
&=\mP_k \mJ_k \mP_k^H
=\begin{bmatrix}
  -i(b\cdot k) & 0 & -|k| \\
  0 & -i(b\cdot k) & 0 \\
  |k| & 0 & 0
 \end{bmatrix} \ , \quad
 \\
\widecheck \mR_k
&=\mP_k \mR_k \mP_k^H
=\mR_k \ , \qquad
\widecheck \mE
=\mP_k \mE \mP_k^H
=\mE \ , 
\end{align}
\end{subequations}
where
\begin{equation} \label{ex:Stokes:Pk}
\mP_k
=\tfrac1{|k|}
 \begin{bmatrix}
  k_1 & k_2 & 0 \\
  -k_2 & k_1 & 0 \\
  0 & 0 & i|k|
 \end{bmatrix} \ .
\end{equation}
The evolution of~\eqref{DAE:Stokes} translates via $y_k =[y_{k,1},y_{k,2},y_{k,3}]^\top :=\mP_k \xk$ 
into the staircase form
\begin{equation} \label{DAE:Stokes:SF}
 \widecheck \mE \dot y_k (t) =(\widecheck \mJ_k -\widecheck \mR_k) y_k(t) \ , \quad t\geq 0 \ .
\end{equation}
Note that in the staircase form~\eqref{staircase:EJR} the third and fifth blocks are missing, i.e. $n_3=0$ and $n_5=0$.
Setting $y_k:=\mP_k w_k$, $y_k=:[y_{k,1}, y_{k,2}, y_{k,3}]^\top$ and using~\eqref{DAE:Stokes:SF}, we obtain
\[
 y_{k,1} =\tfrac1{|k|} k\cdot \phi_k =0 \,, \qquad
 y_{k,3} = i p_k =0 \qquad
 \text{for all } k\ne 0 .
\]
Hence, following \eqref{underlyingODE:y} in Appendix~\ref{app:DAEs:hypocoercivity}, the dynamic part of
the DAE~\eqref{DAE:Stokes:SF} is given as
\begin{equation} \label{DAE:Stokes:y2}
 \dot y_{k,2} =-i(b\cdot k) y_{k,2}-\nu |k|^2 y_{k,2} \ , \quad k\ne 0 \ .
\end{equation}
Therefore, the DAE systems~\eqref{DAE:Stokes:SF}, for $k\ne 0$, exhibit non-trivial dynamics with HC-index $0$, see Definition~\ref{def:DAE:mHC}.
Equation~\eqref{DAE:Stokes:y2} is the modal decomposition of the (dissipative) transport-diffusion equation \eqref{trans-diff} on $\T^2$.
Hence, the solution $(u(\cdot,t),\,p(\cdot,t))$ of the Oseen equation~\eqref{eq:Oseen} converges, as $t\to \infty$, to the constant (in $x$ and $t$) equilibrium $(\phi_0,p_0)\in\R^3$, for $p_0=0$ with the exponential decay rate~$\mu=\min_{k\ne 0} (\nu|k|^2) =\nu$.

The transformation $y_k :=\mP_k w_k$ is related to the Helmholtz-Leray decomposition of vector fields on bounded domains, see~\cite[\S II.3]{FMRT01}:
\begin{remark}[Helmholtz-Leray decomposition]
Note that space-periodic vector fields $u\in H_{per}(\diver,\T^2)$ with $H_{per}(\diver, \T^2):=\{u\in (L^2(\T^2))^2\, |\,\diver u\in L^2(\T^2)\}$ can be written as
\begin{equation} \label{Helmholtz-Leray}
 u =\nabla q +v \qquad\text{with } v\in H_{per}(\diver 0,\T^2)\ ,\quad q\in H^1(\T^2)\ .
\end{equation}
For $u\in H_{per}(\diver,\T^2)$, the expression~\eqref{Helmholtz-Leray} is called the \emph{Helmholtz-Leray decomposition} of~$u$, and~$q$ is unique (up to an additive constant).
Moreover, the map $\Pi_L: H_{per}(\diver,\T^2)\to H_{per}(\diver 0,\T^2)$, $u\mapsto v[u]$ is well-defined and a projection onto the divergence-free vector fields.
The map $\Pi_L$ is called the \emph{Leray projector} [for space-periodic vector fields].
\newline
To determine the Leray projector for $u\in H_{per}(\diver,\T^2)$ in the Fourier representation~\eqref{Fourier}, we extract the leading $2\times2$-subblock from $\mP_k$, $k\in\Z^2\setminus\{0\}$ and define
\begin{equation} 
\widecheck\mP_k
:=\tfrac1{|k|}
 \begin{bmatrix}
  k_1 & k_2 \\
 -k_2 & k_1 \\
 \end{bmatrix} \ , \qquad
 k\in\Z^2\setminus\{0\}\ .
\end{equation}
Then, for the modes $k\in\Z^2\setminus\{0\}$, the Leray projector $\Pi_L$ is given as
\[
 \begin{bmatrix} \phi_{k,1} \\ \phi_{k,2} \end{bmatrix}
 \mapsto 
 \widecheck\mP_k^* 
 \begin{bmatrix} 0 & 0 \\ 0 & 1 \end{bmatrix}
 \widecheck\mP_k
 \begin{bmatrix} \phi_{k,1} \\ \phi_{k,2} \end{bmatrix}\ .
\]
\end{remark}

\makeorange{
\begin{remark}[Possible extensions]
\begin{enumerate}
 \item
Our hypocoercivity analysis can be extended to isotropic Oseen equations on the 3D torus: 
A modal decomposition yields again a family of decoupled DAE-systems (but) with more complicated system matrices.
See also the hypocoercivity analysis of BGK-type equations on 1D, 2D, and 3D tori in~\cite{AAC18}.
 \item
In case of the isotropic Oseen equation~\eqref{eq:Oseen} in the whole space and in exterior domains, solutions exhibit only algebraic-in-time decay, see~\cite{BaJi08}.
For kinetic equations in whole space and without confinement, the hypocoercivity analysis has been adapted to prove the expected algebraic [instead of exponential] temporal decay, see~\cite{BoDoMiMoSc20,ArDoScWo21}.
Similarly, the hypocoercivity analysis can be extended to Oseen equations in whole space.
\end{enumerate}
\end{remark}
}

When modifying~\eqref{eq:Oseen} into an anisotropic Oseen equation with viscosity only in the $x_2$-direction, the dynamics becomes more interesting: 
Depending on the prescribed convection field~$b$, the dynamics may be hypocoercive or not. 
This is the topic of the next section.

\section{Anisotropic Oseen-type equations on the 2D torus}
\label{ssec:Oseen:anisotropic}
\makeorange{Motivated by rotating fluids in geophysics}, the 3D Navier-Stokes equation with anisotropic viscosity, where the vertical viscosity was reduced or even zero, has been studied, \makeorange{see~\cite{CDGG,Pa,Pa05periodic,CheDesGalGre06}}.
In analogy, we shall study now Oseen-type models~with anisotropic viscosity on the torus (\makeorange{similar to \cite{TW96,CheDesGalGre06,Mik23}}).
We consider these equations as a useful simplified mathematical model to illustrate (hypo)coercivity, rather than considering them for their physical interest (similar to~\cite{FMRT01,T95,T97}).
For simplicity of the presentation we again confine ourselves here to the 2D case, i.e. to $\T^2 :=(0,2\pi)^2$:
The model
\begin{subequations}\label{model:Oseen:anisotropic}
\begin{align}
 u_t 
&=-(b(x)\cdot \nabla) u -\nabla p +\nu \partial_{x_2}^2 u \,, 
&& t>0 ,\ \text{on }\T^2 , \label{model:Oseen:anisotropic:1a} \\
 0
&=-\diver u\ , 
&& t\geq 0\ , \label{model:Oseen:anisotropic:1b}
\end{align}
\end{subequations}
is subject to periodic boundary conditions;
it has prescribed transport with a drift velocity vector $b(x)\in\R^2$ which may depend on $x\in\T^2$, and diffusion only in~$x_2$ (hence, we call the  model~\eqref{model:Oseen:anisotropic} \emph{anisotropic}).

\makeorange{
\begin{remark}
Navier-Stokes equation with anisotropic viscosity in~3D, are used to model rotating fluids in geophysics~\cite{CDGG,Pa,Pa05periodic,CheDesGalGre06}; but also to model \emph{anisotropic fluids}, e.g.~with
relaxed ellipticity condition on the viscosity tensor~\cite{Mik23}. 
Linearizing the Navier-Stokes equation with anisotropic viscosity around a constant vector field~$b\in\R^2$ yields the anisotropic Oseen equation~\eqref{model:Oseen:anisotropic}.

Moreover, Oseen equations~\eqref{eq:Oseen} with space-dependent vector field $b=b(x)$ have been derived to describe the motion of a Navier–Stokes liquid around a rotating rigid body, see~\cite[\S VIII]{Gal11}.
This is strictly speaking not an Oseen equation anymore, since the derivation of Oseen needs that $b$ is constant in space.

The periodic setting may be used to model fluid flow in periodic composite structures, and is a basic setting in homogenisation theories for inhomogeneous fluids, see~\cite{Mik23}.
\end{remark}
}
We analyze this model in two variants: For a constant convection field~$b\in\R^2$, the modes still decouple but the generator of the evolution is (depending on $b$) either coercive or not even hypocoercive.
If the convection field is non-constant in space, e.g. if
$b=[\sin (x_2),0]^\top$, then 
the spatial modes are coupled and the generator of the (infinite dimensional) problem becomes hypocoercive.

\subsection{Oseen-type equation with constant drift velocity}
\label{ssec:const-drift}

Let us first consider \eqref{model:Oseen:anisotropic} as an evolution equation on the space of divergence-free vector fields. 
In the case of a constant drift velocity vector~$b\in\R^2$, we can argue like in~\S\ref{ssec:Stokes} to find $p(\cdot,t)=\textrm{const}$, and we  again normalize the pressure as $p\equiv 0$.
Then, the (linear) generator of~\eqref{model:Oseen:anisotropic} takes the shape $-\mC =-b\cdot \nabla +\nu \partial_{x_2}^2$ and it acts identically on the two components $\uI$ and $\uII$. 
Proceeding as in \S\ref{ssec:Stokes} and using the analog of the modal evolution equation \eqref{eq:Stokes:Fourier} shows that, for any initial condition $u(0)\in H_{per}(\diver 0, \T^2)$,  Equation \eqref{model:Oseen:anisotropic} with $b\in\R^2$ has a unique mild solution for $t>0$. 
Hence, the operator~$-\mC$ generates a $C_0$-semigroup on~$\cH :=H_{per}(\diver 0, \T^2)$, see also \cite{EnNa00,Pazy}.
But due to the degenerate, anisotropic diffusion and the lack of hypocoercivity (see Proposition \ref{prop:Oseen:constantU:notHC} below), solutions are in general not smooth here. 
This is illustrated by an example, see~\eqref{counterex:Oseen:constantU:notHC} below.

In order to check the hypocoercivity of~$\mC$ as characterized in~\eqref{exp-decay}, we determine the kernel of $\mC$ in~$\cH$:
\begin{lemma}\label{lem:Oseen:constantU:kernel}
Let $b=[b_1,b_2]^\top\in\R^2$ be constant in~\eqref{model:Oseen:anisotropic}.
\begin{itemize}
 \item 
If $b_1\ne 0$ then $\ker\mC =\R^2$, i.e. the constant-in-$x$ flows.
Moreover, the $L^2$-orthogonal complement of $\ker\mC$ in~$\cH$ is given by \[ \{u\in H_{per}(\diver 0, \T^2) \ |\ \int_{\T^2} u\dd[x]=0\} \ .\]
 \item
If $b_1=0$ then the kernel of~$\mC$ in~$\cH$ is given as 
\[ 
\ker\mC =\{ u\in\cH:\ u_1=\const\ ,\quad u_2=u_2(x_1)\}. 
\]
Moreover, the $L^2$-orthogonal complement of $\ker\mC$ in~$\cH$ is given by 
\begin{equation} \label{tildeH:b_1:0} 
\{u\in H_{per}(\diver 0, \T^2) \ |\ \int_{\T^2} u\dd[x]=0\ ,\quad u_2=u_2(x_2) \}.
\end{equation}
\end{itemize}
\end{lemma}
\begin{proof}
If $b_1\ne0$ then $\mC u=0$ reads $b\cdot \nabla u_j =\nu \partial_{x_2}^2 u_j$; $j=1,2$.
The space-``time'' periodic parabolic equation $b_1 \partial_{x_1} v =-b_2 \partial_{x_2} v +\nu \partial_{x_2}^2 v$ admits only constant periodic solutions.
The Hilbert space~$\cH$ is endowed with an $L^2$-inner product, which implies the final statement.

If $b_1=0$ then $\mC u=0$ reads $0 =-b_2 \partial_{x_2} u_j +\nu \partial_{x_2}^2 u_j$; $j=1,2$.
The elliptic equation $0 =-b_2 \partial_{x_2} v +\nu \partial_{x_2}^2 v$ admits only periodic solutions which are constant in $x_2$ and $2\pi$-periodic in $x_1$.
Moreover, $0 =\diver u =\partial_{x_1}u_1(x_1) +\partial_{x_2}u_2(x_1) =\partial_{x_1}u_1(x_1)$ implies that $u_1 =\const$.
The Hilbert space~$\cH$ is endowed with an $L^2$-inner product, which implies the final statement.
\end{proof}
We continue to analyze the anisotropic Oseen-type model~\eqref{model:Oseen:anisotropic} with $b\in\R^2$ and $b_1\ne 0$.
Following the characterization of hypocoercivity in~\eqref{exp-decay} and Lemma~\ref{lem:Oseen:constantU:kernel}, we also introduce the Hilbert space 
\[
 \wcH :=\{ u\in\cH\ |\ \int_{\T^2} u \dd[x]=0 \},
\]
endowed with the $L^2$-inner product. Since the condition $\int_{\T^2} u \dd[x]=0$ is preserved under the flow, $-\mC$ also generates a $C_0$-semigroup on $\wcH$, see e.g. \cite{Pazy,EnNa00}. 
But related to its long-time behavior we have the following result:
\begin{proposition}\label{prop:Oseen:constantU:notHC}
Let $b\in\R^2$ be constant with $b_1\ne 0$.
Then, the operator~$\mC =b\cdot\nabla -\nu \partial_{x_2}^2$ is neither coercive nor hypocoercive in~$\wcH$.
\end{proposition}
\begin{proof}
First we introduce the maximal domain of $\mC$ in $\wcH$: $\cD(\mC)=\{u\in\wcH\ | \ \mC u\in (L^2(\T^2))^2\}$. Since $\mC$ has constant coefficients, $u\in \cD(\mC)$ implies $\mC u\in\wcH$. For future reference we also give their characterization in Fourier space (cf.\ \eqref{Fourier} and~\cite[\S II.5]{FMRT01}):
\begin{align}\label{D(C)-fourier}
  u\in \wcH \quad &\Leftrightarrow \quad 
  \{\phi_k\} \in \ell^2(\Z^2;\C^2)\;\mbox{ with } \phi_0=0,\; k\cdot \phi_k=0\;
  \mbox{for all}\ k\in\Z^2, \nonumber 
\\
  u\in \cD(\mC) \quad &\Leftrightarrow \quad 
  \{\phi_k\},\; \big\{\big(i(b\cdot k)+\nu k_2^2\big)\phi_k\big\} \in \ell^2(\Z^2;\C^2)\; \nonumber
\\
  &\hspace{8em} \mbox{ with } \phi_0=0,\; k\cdot \phi_k=0\;
 \mbox{for all}\ k\in\Z^2.
\end{align}

To investigate the coercivity of $\mC$, we compute for $u\in\cD(\mC)$:
\begin{equation}\label{C-coercive}
  \ip{u}{\mC u}_{\wcH}
 =\nu \int_{\T^2} |\partial_{x_2} u|^2 \dd[x]
 =\nu \int_{\T^2} \big[ (\partial_{x_2} \uI)^2 +\tfrac12 (\partial_{x_1} \uI)^2 +\tfrac12 (\partial_{x_2} \uII)^2 \big] \dd[x],
\end{equation}
where we have used that $\diver u=0$.
The last integral does not involve $\partial_{x_1} \uII$.
Hence, $u:=[0,\sin (x_1)]^\top$ is a counterexample to coercivity on~$\wcH$ since $\ip{u}{\mC u}_{\wcH} =0$.

If~$\mC$ was hypocoercive, the trivial solution~$u=0$ would be asymptotically stable on~$\wcH$ due to the decay estimate~\eqref{exp-decay}.
But for any $v\in H^1_{per}(\T)$ with $\int_\T v(x_1)\dd[x_1]=0$, the (undamped) traveling wave 
\begin{equation} \label{counterex:Oseen:constantU:notHC}
 u(x,t) =\begin{bmatrix} 0 \\ v(x_1 -b_1 t) \end{bmatrix}
\end{equation}
solves~\eqref{model:Oseen:anisotropic} in~$\wcH$, in violation of~\eqref{exp-decay}.
\end{proof}

This lack of hypocoercivity in the model~\eqref{model:Oseen:anisotropic} with constant $b\in\R^2$ and $b_1\ne 0$ can be understood quite easily:
It includes drift and diffusion in the $x_2$-direction, but the uniform transport in the $x_1$-direction does not entail mixing between the different vertical layers.
If the flow field has only a vertical component~$\uII$, possibly different in each vertical layer (as shown in~\eqref{counterex:Oseen:constantU:notHC}), then the flow field gets transported in the $x_1$-direction and remains incompressible.

The lack of hypocoercivity can also be verified by considering \eqref{model:Oseen:anisotropic} as a constrained partial differential equation (PDAE), and bringing its modal representation into staircase form, see Example~\ref{ex:Oseen:constantU} in Appendix~\ref{app:DAEs:hypocoercivity}.

\begin{remark}
In contrast to the above case $b_1\ne0$, for the anisotropic Oseen-type model~\eqref{model:Oseen:anisotropic} with $b\in\R^2$ and $b_1 =0$, we have to define the Hilbert space $\wcH$ as in~\eqref{tildeH:b_1:0} endowed with the $L^2$-inner product.
Using that $u\in\wcH$ satisfies $u_2=u_2(x_2)$ and using~\eqref{C-coercive}, this shows that $\ip{u}{\mC u}_{\wcH} \geq \nu \|u\|_{\dot{H}^1}^2 /2$ for all $u\in\cD(\mC)=\{u\in\wcH\,|\, \partial_{x_2}^2u\in (L^2(\T^2))^2\}$.
Due to the Poincar\'e inequality $\|u\|_{\dot{H}^1}\geq \|u\|_{L^2}$ for all $u\in\wcH\cap H^1(\T^2)$ (trivially obtained from the Fourier representation \eqref{Fourier}; see also~\cite{Ev10,FMRT01}), $\mC$ is coercive on (the smaller) Hilbert space $\wcH$ here.
\end{remark}

\subsection{Oseen-type equation with non-constant drift velocity}
\label{ssec:nonconst-drift}

As second example we consider the anisotropic  equation~\eqref{model:Oseen:anisotropic} with a non-constant drift field $b(x) =[b_1(x_2),0]^\top$ such that  
\begin{subequations}\label{model:Oseen:anisotropic_2}
\begin{align}
 u_t 
&=-b_1(x_2)\ \partial_{x_1} u -\nabla p +\nu \partial_{x_2}^2 u \,, 
&& t>0 ,\ \text{on }\T^2 , \label{model:Oseen:anisotropic_2:1a} \\
 0
&=-\diver u\ , 
&& t\geq 0\ . \label{model:Oseen:anisotropic_2:1b}
\end{align}
\end{subequations}
In order to simplify the detailed hypocoercivity analysis below, 
we choose $b_1(x_2) =\sin (x_2)$; in this case only neighboring $k_2$-modes are coupled. 
We conjecture that choosing another non-constant $x_2$-periodic drift field $b_1$ would lead to an analogous result, but it would need a more cumbersome analysis.  
\makeorange{In this direction, the existence of stationary solutions for the isotropic Oseen equation~\eqref{eq:Oseen} with space-dependent drift~$b=b(x)$ and $x\in\R^3$ has been studied, see~\cite{AmrCon11}.}

In contrast to \S\ref{ssec:const-drift}, this model does not preserve vertical layers of the flow field but rather mixes them, hence giving rise to hypocoercivity.
While the pressure could have been eliminated from the beginning in the model with constant drift term, this is not possible here.
Hence, \eqref{model:Oseen:anisotropic_2} with $b_1(x_2)$ has to be considered as a (true) constrained PDE.\\

\noindent
{\bf Analytic framework and well-posedness:}
For the analysis of this case we apply the divergence to the evolution equation~\eqref{model:Oseen:anisotropic_2:1a}, leading to the following condition on the pressure:
\begin{equation}\label{model:divOseen}
 \begin{cases}
  \Delta p =-\cos(x_2)\ \partial_{x_1} \uII , \quad \text{on } \T^2 , \\
  \text{periodic boundary conditions for $p$} , \\
  \int_{\T^2} p \dd[x] =0 ,
 \end{cases}
\end{equation}
where the last condition was added to ensure uniqueness of~$p$. 
For a given inhomogeneity $u\in (L^2(\T^2))^2$, Equation \eqref{model:divOseen} has a unique weak solution in the homogeneous Sobolev space $\dot H^1_{per}(\T^2)$, which we denote by $p =p[u]$. A standard elliptic estimate shows
\begin{equation}\label{p-u-estimate}
    \|p\|_{\dot H^1_{per}(\T^2)} \le \|u_2\|_{L^2(\T^2)/ \R} \le  \|u\|_{(L^2(\T^2)/ \R)^2}.
\end{equation}

The generator of the flow~\eqref{model:Oseen:anisotropic_2} with $b_1(x_2) =\sin (x_2)$ takes the form
\begin{equation}\label{Oseen:generator:sin}
 -\mC u 
 =-\sin(x_2)\ \partial_{x_1} u -\nabla p[u] +\nu \partial_{x_2}^2 u .
\end{equation}
Note that $\diver u=0$ and $p[u]$ from \eqref{model:divOseen} imply that $\diver(\mC u)=0$. The assertion $\int \mC u\dd[x]=0$ follows trivially from $\int u \dd[x]=0$.
As before, we choose the Hilbert spaces
\[
 \cH :=H_{per}(\diver 0, \T^2) ,\qquad
 \wcH :=\{ u\in\cH\ |\ \int_{\T^2} u \dd[x]=0 \},
\]
and the maximal domain of $\mC$ is $\cD(\mC)=\{u\in\wcH\ | \ \sin(x_2)\partial_{x_1} u-\nu\partial^2_{x_2}u \in (L^2(\T^2))^2\}$, since the linear map $u\mapsto \nabla p$ from \eqref{model:divOseen} is bounded on $(L^2(\T^2))^2$. 

With this framework we have the following result:
\begin{lemma}\label{lem:Oseen:sin}
Let $b_1(x_2)=\sin (x_2)$ in~\eqref{model:Oseen:anisotropic_2}.
Then:
\begin{enumerate}[(1)]
\item \label{Oseen:Ux2:kernel}
$\ker\mC =\R^2$, i.e. the constant-in-$x$ flows; and the $L^2$-orthogonal complement of $\ker\mC$ is $\wcH$.
\item \label{Oseen:Ux2:notC}
The operator~$\mC$ is not coercive on~$\wcH$.
\end{enumerate}
\end{lemma}
\begin{proof}
\ref{Oseen:Ux2:kernel}
For $u\in\cD(\mC)$ we compute:
\[
0 
=\int_{\T^2} u\cdot \mC u\dd[x]
=\int_{\T^2} \big(\tfrac12 \sin(x_2)\ \partial_{x_1} |u|^2 +\nu |\partial_{x_2} u|^2\big) \dd[x]
=\nu \int_{\T^2} |\partial_{x_2} u|^2 \dd[x] ,
\]
to obtain $u=u(x_1)$, and hence $0=\diver u=\partial_{x_1} \uI$.
So, using $\uI=\textrm{const}$ in~\eqref{Oseen:generator:sin} yields $p=p(x_2)$ and 
\[
 \sin(x_2)\ \partial_{x_1} \uII(x_1) +\partial_{x_2} p(x_2) =0 ,
\]
or equivalently for $x_2\not\in\{0,\pi\}$:
\[
 \partial_{x_1} \uII(x_1)
 =-\frac{\partial_{x_2} p(x_2)}{\sin (x_2)}
 =\lambda
 \qquad 
 \text{with some } \lambda\in\R .
\]
The periodicity of $\uII$ implies $\lambda=0$ and hence $\uII=\textrm{const}$, $p=0$.

\ref{Oseen:Ux2:notC}
As in \S\ref{ssec:const-drift}, $u:=[0, \sin (x_1)]^\top \in \cD(\mC)$ is a counterexample to coercivity on~$\wcH$, since $\ip{u}{\mC u}_{\wcH} =0$, by noting that $p[u]=\frac12 \cos(x_2)\cos(x_1)$ solves \eqref{model:divOseen}.
\end{proof}

Next we discuss the existence and uniqueness of a mild solution to the initial value problem \eqref{model:Oseen:anisotropic_2} when eliminating the pressure via \eqref{model:divOseen}. The resulting PDE-evolution problem for $u$ only, with $u(0)\in\wcH$, will be analyzed with semigroup theory, see e.g.\ \cite{Pazy}. 

\begin{proposition}
    On $\wcH$, the operator $-\mC$ from \eqref{Oseen:generator:sin} satisfies the following properties:
    \begin{enumerate}
        \item[(1)] It is dissipative, densely defined, and closed.
        \item[(2)] It generates a $C_0$-semigroup of contractions on $\wcH$.
    \end{enumerate}
\end{proposition}

\begin{proof}
(1) A similar computation as in \eqref{C-coercive} shows $\langle u,-\mC u\rangle_{\wcH} \le 0\;\mbox{ for all}\, u \in\cD(\mC)$, i.e.\ the dissipativity of~$-\mC$. 

For the density of $\cD(\mC)$ in $\wcH$ we first consider
\[
\widetilde\cD := \{u\in \wcH\, | \,\partial_{x_1}u,\,\partial_{x_2}^2u \in (L^2(\T^2))^2\} \subset\cD(\mC).
\]
Truncating the Fourier representation of any $u\in\widetilde\cD$ (in analogy to \eqref{D(C)-fourier}) one easily finds that $\widetilde\cD$ is $L^2$-dense in $\wcH$, and hence also $\cD(\mC)$ in $\wcH$.

Using Theorem 1.4.5(c) of \cite{Pazy} we find that $-\mC$ is closable. Moreover, $-\overline\mC=-\mC$, since $\cD(\mC)$ was chosen as the maximal domain. 

(2) For the second statement we easily compute 
$$
  -\mC^* u 
 =\sin(x_2)\ \partial_{x_1} u -\nabla p[u] +\nu \partial_{x_2}^2 u,  
$$
which is defined at least on $\widetilde\cD$ and also dissipative. Hence, Corollary 1.4.4 (to the Lumer-Phillips Theorem) from \cite{Pazy} implies that $-\mC$ is the infinitesimal generator of a $C_0$-semigroup of contractions on~$\wcH$.
\end{proof}

Next we shall illustrate the short-time behavior of $\|e^{-\mC t}\|$ (in the spirit of~\eqref{short-t-decay}) by analyzing the Taylor expansion of the norm for a single trajectory at $t=0$. Using the same initial condition as in the proof of Lemma \ref{lem:Oseen:sin}, i.e.\ $u(0)=[0,\sin(x_1)]^\top$ which is a linear combination of the modes $\binom{\pm1}{0}$~, we compute:
\begin{align*}
    \mC u(0)&=-\frac12 \begin{bmatrix} \sin(x_1)\cos(x_2) \\ -\cos(x_1)\sin(x_2) \end{bmatrix}
    \in \cD(\mC),\\
    \mC^2 u(0)&=-\frac12 \begin{bmatrix} \frac35 \cos(x_1)\sin(2x_2)+\nu\sin(x_1)\cos(x_2) \\ \sin(x_1)\big(\frac12-\frac3{10}\cos(2x_2)\big)-\nu\cos(x_1)\sin(x_2) \end{bmatrix}
    \in \cD(\mC),\\
    \mC^3 u(0)&= \begin{bmatrix} \sin(x_1) \big(\frac{19}{80}\cos(x_2)-\frac{63}{400}\cos(3x_2)\big) \\ \cos(x_1)\big(\frac{21}{400}\sin(3x_2)-\frac{19}{80}\sin(x_2)  \big) \end{bmatrix}\\
    &\qquad +\nu  \begin{bmatrix} -\frac32 \cos(x_1)\sin(2x_2) \\ \sin(x_1)\big(-\frac14+\frac34\cos(2x_2)\big) \end{bmatrix}
    +\frac{\nu^2}2 \begin{bmatrix} -\sin(x_1)\cos(x_2) \\ \cos(x_1)\sin(x_2) \end{bmatrix},
\end{align*}
where we used the following solutions to the corresponding Poisson equation~\eqref{model:divOseen}:
\begin{align*}
    p[\mC u(0)]&=-\frac1{20} \sin(x_1)\sin(2x_2),\\
    p[\mC^2 u(0)]&=-\frac1{80} \cos(x_1)\big(7\cos(x_2)-\frac35\cos(3x_2)\big)-\frac\nu{20} \sin(x_1)\sin(2x_2).
\end{align*}
Hence we obtain
\begin{align*}
    \|u(0)\|^2_{\wcH} &=2\pi^2,\\
    \ddt \|u(t)\|^2_{\wcH}\,\big|_{t=0} &= -2\langle u(0),\mC u(0)\rangle_{\wcH}=0,\\
    \ddtddt \|u(t)\|^2_{\wcH}\,\big|_{t=0} &= 2\langle \mC u(0),\mC u(0)\rangle_{\wcH} +2\langle u(0),\mC^2 u(0)\rangle_{\wcH}=0,\\
    \ddtddtddt \|u(t)\|^2_{\wcH}\,\big|_{t=0} &= -6\langle \mC u(0),\mC^2 u(0)\rangle_{\wcH} -2\langle u(0),\mC^3 u(0)\rangle_{\wcH}=-2\nu\pi^2<0.
\end{align*}
This implies the following lower bound on the propagator norm:
\begin{equation}\label{prop-norm-est}
 \|e^{-\mC t}\| := \sup_{\|v\|_{\wcH}=1} \|e^{-\mC t}v\|_{\wcH}\ge1 -\frac1{12}\nu\,t^3+\mathcal O(t^4) \quad \mbox{for } t\to 0^+.
 \end{equation}
The first non-constant term in the Taylor expansion of $\|u(t)\|^2_{\wcH}$ has exponent~$a=3$.
In Proposition \ref{prop:Oseen:y2} below, we show that the modal expansion of~\eqref{model:Oseen:anisotropic_2} with~$b_1(x_2)=\sin (x_2)$,
has hypocoercivity index $m_{HC}=1$. 
Although the analog of the result \eqref{short-t-decay} has not yet been established for unbounded generators~$-\mC$, we note that the exponent $a=3$ in the Taylor series of~$\|u(t)\|^2_{\wcH}$ satisfies again $a=2\mHC+1$.\\ 

\noindent
{\bf Modal decomposition:}
To analyze the hypocoercivity of~\eqref{model:Oseen:anisotropic_2} with $b_1(x_2) =\sin (x_2)$, we perform a modal decomposition of this model as in~\S\ref{ssec:Stokes}.
But due to the non-constant coefficient $b_1(x_2)$, these modes decouple now only w.r.t.~$k_1$, but not w.r.t.~$k_2$.
As the analog of~\eqref{eq:Stokes:Fourier} we obtain for $k\in\Z^2$ the DAE system
\begin{equation}\label{eq:Oseen:Fourier}
\begin{cases}
\ddt \phi_k 
&= \tfrac{k_1}{2}(\phi_{k+e_2} -\phi_{k-e_2}) -i k p_k -\nu k_2^2 \phi_k \ , 
\quad 
t>0 \ , \\
0& =-i k\cdot \phi_k  \ ,
\end{cases}
\end{equation}
with $e_2:=[0,\, 1]^\top$.
All modes with $k_1=0$ are fully decoupled, showing that $\phi_0(t)=\const$ and one has exponential decay with rate at least~$\mu=\nu$ for all modes with $k=(0,k_2)\ne 0$. 
More precisely, proceeding as for \eqref{eq:Stokes:Fourier} with $b=0$ one obtains: $\phi_{(0,k_2),2}(t)=0,\,p_{(0,k_2)}(t)=0$ for $k_2\ne0$, as well as
\begin{equation}\label{k10-decay}
  |\phi_{(0,k_2),1}(t)| \le |\phi_{(0,k_2),1}(0)|\,e^{-t\nu k_2^2}\,,
  \quad k_2\ne0\,.
\end{equation}

But for $k_1\ne 0$ the modes are only semi-decoupled, i.e., for each fixed $k_1\ne 0$, \eqref{eq:Oseen:Fourier} forms an infinite dimensional DAE system with $k_2\in\Z$.
Using the Fourier decomposition it is
notationally simpler to consider  $\ell^2(\Z)$ rather than $\ell^2(\N)$.
Then the doubly infinite complex vector 
\begin{align*}
w_{k_1} 
&:= 
\big\{[\phi_{(k_1,k_2),1}, \phi_{(k_1,k_2),2}, p_{(k_1,k_2)}]^\top \big\}_{k_2\in\Z} \\
& = [...\,|\, \phi_{(k_1,-1),1}, \phi_{(k_1,-1),2}, p_{(k_1,-1)}\,|\,\phi_{(k_10),1}, \phi_{(k_1,0),2}, p_{(k_1,0)}\,|\,\phi_{(k_1,1),1}, \phi_{(k_1,1),2}, p_{(k_1,1)}\,|\,...]^\top \\
&\in\ell^2(\Z;\C^3) 
\end{align*}
satisfies the DAE
\begin{equation} \label{DAE:Oseen}
\mE \dot{w}_{k_1}(t)
=
(\mJ_{k_1} -\mR) w_{k_1}(t) ,
\qquad
t\geq 0 ,
\end{equation}
with doubly infinite diagonal matrices $\mE, \mR\in\cL(\ell^2(\Z;\C^3))$ given as
\begin{align}
\mE
&:= \diag(\ \diag([1,1,0]);\, k_2\in\Z) \nonumber 
\\ 
&\qquad = \diag(\ldots\,|\, 1,1,0\,|\, 1,1,0\,|\, 1,1,0\,|\, \ldots)\ , \label{DAE:Oseen:E}
\\
\mR
&:= \nu \diag(\ \diag([k_2^2,k_2^2,0]);\, k_2\in\Z) \nonumber 
\\ 
&\qquad = \nu \diag(\ldots\,|\, 4,4,0\,|\, 1,1,0\,|\, 0,0,0\,|\, 1,1,0\,|\, 4,4,0\,|\, \ldots)\ , 
\label{DAE:Oseen:R}
\end{align}
and the doubly infinite block-tridiagonal matrix $\mJ_{k_1}\in\cL(\ell^2(\Z;\C^3))$ with $3\times3$-blocks given as 
\begin{equation}\label{DAE:Oseen:J}
\begin{split}
&\mJ_{k_1}
:= 
\\
&\left[ \begin{array}{ccc|ccc|ccc|ccc|ccc}
 & \ddots & & & & & & & & & & & & & \\
\hline
-\tfrac{k_1}2 & & & 0 & 0 & -i k_1 & \tfrac{k_1}2 & & & & & & & & \\
 & -\tfrac{k_1}2 & & 0 & 0 & i & & \tfrac{k_1}2 & & & & & & & \\
 & & 0 & -i k_1 & i & 0 & & & 0 & & & & & & \\
\hline
 & & & -\tfrac{k_1}2 & & & 0 & 0 & -i k_1 & \tfrac{k_1}2 & & & & & \\
 & & & & -\tfrac{k_1}2 & & 0 & 0 & 0 & & \tfrac{k_1}2 & & & & \\
 & & & & & 0 & -i k_1 & 0 & 0 & & & 0 & & & \\
\hline
 & & & & & & -\tfrac{k_1}2 & & & 0 & 0 & -i k_1 & \tfrac{k_1}2 & & \\
 & & & & & & & -\tfrac{k_1}2 & & 0 & 0 & -i & & \tfrac{k_1}2 & \\
 & & & & & & & & 0 & -i k_1 & -i & 0 & & & 0 \\
\hline
 & & & & & & & & & & & & & \ddots & 
\end{array}\right] 
{\scriptsize
\begin{array}{l}
 k_2=-1 \\[3em] k_2=0 \\[3em] k_2=1
\end{array} } \\
& {\scriptstyle\hspace{2em} \widetilde k_2=-2 \hspace{4em} \widetilde k_2=-1 \hspace{5em} \widetilde k_2=0 \hspace{5em} \widetilde k_2=1 \hspace{4em} \widetilde k_2=2 }
\end{split},
\end{equation}
where missing elements 
are understood to be zero.
Here we label the $3\times 3$ submatrices $\mJ_{k_1}(k_2,\widetilde k_2)$ by the corresponding modal numbers of $k_2$.
In particular, the diagonal blocks have the general form
\begin{equation}\label{DAE:Oseen:Jkk}
\mJ_{k_1} (k_2,k_2)
=
\begin{bmatrix}
 0 & 0 & -i k_1 \\
 0 & 0 & -i k_2 \\
-i k_1 & -i k_2 & 0 
\end{bmatrix} .
\end{equation}
\medskip

\noindent
{\bf Staircase transformation:}
To bring~\eqref{DAE:Oseen} into staircase form
we have to modify Step~1 in~\cite[Algorithm 5]{AAM21}, see e.g.~Appendix~\ref{app:DAEs:hypocoercivity}, and hence also the presentation of the result in Lemma~\ref{lem:tSF}:

On the one hand the matrix~$\mE$ is already diagonal, but on the other hand it has both an infinite dimensional kernel and range.
Hence, it would be impractical to separate the corresponding eigenvalues.
Instead we shall leave them interlaced as in~\eqref{DAE:Oseen:E}.
This has the following (purely notational) consequence on partitioning the matrices $\mJ_{k_1}$, $k_1\ne 0$. 
With the notation from~\cite[Algorithm 5: Step 1.2]{AAM21}, see also Appendix~\ref{app:DAEs:hypocoercivity}, we  partition every $3\times 3$ submatrix of $\mJ$ into blocks of size~2 and~1.
In \cite[Algorithm 5: Step 1.2]{AAM21} the finite dimensional matrix $\mJ$ is partitioned as  $\begin{bmatrix} \widetilde\mJ_{1,1} & -\widetilde\mJ_{2,1}^H \\ \widetilde\mJ_{2,1} & \widetilde\mJ_{2,2}  \end{bmatrix}$.
Just for notational simplicity we suppressed here and in the sequel the index $k_1$ and write $\mJ:=\mJ_{k_1}$. Next, we decompose the space of sequences $\ell^2(\Z;\C^3)$ as
\begin{equation}
    \ell^2(\Z;\C^3) =\ell^2(\Z;\spn\{e_1,e_2\}) \oplus \ell^2(\Z;\spn\{e_3\})  ,
\end{equation}
where $e_1=[1,0,0]^\top$, $e_2=[0,1,0]^\top$, $e_3=[0,0,1]^\top$; and use that these subspaces are isomorphic to 
\begin{equation}
    \ell^2(\Z;\spn\{e_1,e_2\}) \simeq \ell^2(\Z;\C^2)\ , \qquad
    \ell^2(\Z;\spn\{e_3\}) \simeq \ell^2(\Z;\C)\ .
\end{equation}
In analogy to Algorithm~\ref{algorithm:staircase:E_J-R}:~Step~1.2, we decompose the operator~$\mJ\in\cL(\ell^2(\Z;\C^3))$ into the linear operators
\begin{align*}
    \tJ_{1,1}\in \cL\big(\ell^2(\Z;\C^2)\big) , \quad
    \tJ_{2,1}\in \cL\big(\ell^2(\Z;\C^2);\ell^2(\Z;\C)\big) , \quad
    \tJ_{2,2}\equiv 0\in \cL\big(\ell^2(\Z;\C)\big) , 
\end{align*}
represented by the doubly infinite block-tridiagonal matrices with $2\times2$-blocks
\begin{equation}
\begin{split}\label{DAE:Oseen:J11}
\tJ_{1,1}
&:= \frac{k_1}2
\left[ \begin{array}{c|cc|cc|cc|cc|cc|c}
\ddots & & \ddots & & \ddots & & & & & & & \\
\hline
& -1 & & 0 & 0 & 1 & & & & & & \\
& & -1 & 0 & 0 & & 1 & & & & & \\
\hline
& & & -1 & & 0 & 0 & 1 & & & & \\
& & & & -1 & 0 & 0 & & 1 & & & \\
\hline
& & & & & -1 & & 0 & 0 & 1 & & \\
& & & & & & -1 & 0 & 0 & & 1 & \\
\hline
& & & & & & & \ddots & & \ddots & & \ddots
\end{array}\right] 
{\scriptsize
\begin{array}{l}
 \\[0.5em] k_2=-1 \\[2em] k_2=0 \\[2em] k_2=1 \\[1em]
\end{array} } \ , \\
& {\scriptstyle\hspace{5.8em} \widetilde k_2=-2 \hspace{0.9em} \widetilde k_2=-1 \hspace{1em} \widetilde k_2=0 \hspace{1.3em} \widetilde k_2=1 \hspace{0.5em} \widetilde k_2=2 }
\end{split}
\end{equation}
and the doubly infinite block-tridiagonal matrices with $1\times2$-blocks
\begin{equation}
\begin{split}\label{DAE:Oseen:J21}
\tJ_{2,1}
:= \diag( [-i k_1\ -i k_2];\ k_2\in\Z) 
&=\! \left[ \begin{array}{c|cc|cc|cc|c}
\ddots & & & & & & & \\
 \hline
 & -i k_1 & i & & & & & \\
 \hline
 & & & -i k_1 & 0 & & & \\
 \hline
 & & & & & -i k_1 & -i & \\
 \hline
 & & & & & & & \ddots
\end{array}\right] 
{\scriptsize 
\begin{array}{l}
\\[-0.4em] k_2=-1 \\[0.6em] k_2=0 \\[0.6em] k_2=1 \\[0.5em]
\end{array} } . \\
& {\scriptstyle\hspace{4em} \widetilde k_2=-1 \hspace{1em} \widetilde k_2=0 \hspace{2em} \widetilde k_2=1 \hspace{2.5em} } 
\end{split}
\end{equation}

When partitioning $\tR$ like $\tJ$, we find from \eqref{DAE:Oseen:R} that $\tR_{2,2}=0$. Hence Step~2 of Algorithm~\ref{algorithm:staircase:E_J-R} is trivial.
Since the subblocks of $\tJ_{2,1}$ are identical to the $(2,1)$-subblock of~\eqref{DAE:Stokes:J}, with $k_1\ne 0$ fixed and $k_2\in\Z$, their singular value decompositions (SVDs) also coincide.
Thus, the unitary transformation matrix $\mP_{k_1,k_2}$ (pertaining to each $3\times 3$ submatrix~\eqref{DAE:Oseen:Jkk}) is given by~\eqref{ex:Stokes:Pk}.
Hence, the (global) unitary transformation matrix for the DAE systems~\eqref{DAE:Oseen} with $k_1\in\Zo$ is given by the block diagonal matrix
\begin{equation}\label{model:Oseen:Pk1}
\mP_{k_1}
=
\diag\Bigg(
\tfrac1{|k|}
 \begin{bmatrix}
  k_1 & k_2 & 0 \\
  -k_2 & k_1 & 0 \\
  0 & 0 & i|k|
 \end{bmatrix} ;\
k_2\in\Z \Bigg) \ ,
\end{equation}
which represents a bounded unitary operator on $\ell^2(\Z;\C^3)$. Also the leading $2\times2$-subblocks~\eqref{ex:Stokes:Pk} form a bounded unitary operator on $\ell^2(\Z;\C^2)$. For each fixed $k_1\ne0$ in~\eqref{DAE:Oseen}, this is an infinite dimensional analog of the staircase transformation in Lemma \ref{lem:tSF}. Hence it translates~\eqref{DAE:Oseen}, via $y_{k_1} :=\mP_{k_1} w_{k_1} \in\ell^2(\Z;\C^3)$, into the staircase form
\begin{equation}\label{DAE:Oseen:staircase}
\widecheck\mE \dot{y}_{k_1}(t)
=
(\widecheck{\mJ}_{k_1} -\widecheck\mR) y_{k_1}(t) , 
\qquad
t\geq 0 .
\end{equation}
Here, the matrices are given by $\widecheck\mE =\mP_{k_1} \mE\mP_{k_1}^H =\mE$, $\widecheck{\mR} =\mP_{k_1} \mR\mP_{k_1}^H =\mR$, and $\widecheck{\mJ}_{k_1}=\mP_{k_1}\mJ_{k_1}\mP_{k_1}^H$ is a block tridiagonal matrix with the blocks
\begin{align*}
\widecheck{\mJ}_{k_1}(k_2,k_2)
&=
\begin{bmatrix}
 0 & 0 & -|k| \\
 0 & 0 & 0 \\
 |k| & 0 & 0 
\end{bmatrix} , 
 \\
\widecheck{\mJ}_{k_1}(k_2,k_2\pm 1)
&=
\pm \frac{k_1}{2 |k|\ |k\pm e_2|}
\begin{bmatrix}
|k|^2 \pm k_2 & \mp k_1 & 0 \\
\pm k_1 & |k|^2\pm k_2 & 0 \\
0 & 0 & 0 
\end{bmatrix} .
\end{align*}
In analogy to Definition \ref{def:DAE:mHC}, the non-trivial dynamics of~\eqref{DAE:Oseen:staircase} is given by \eqref{underlyingODE:y} as the infinite ODE
\begin{equation}\label{ODE:Oseen:y2}
\mE_{2,2} \big(\dot{y}_{k_1}\big)_2
=
\big(\widehat{\mJ}_{k_1} -\widehat\mR\big)_{2,2} \big(y_{k_1}\big)_2 .
\end{equation}
Here, $\big(y_{k_1}\big)_2=\big\{\frac{1}{|k|}\big(-k_2\ukI+k_1\ukII\big)\big\}_{k_2\in\Z} \in\ell^2(\Z;\C)$ is the projection of $y_{k_1} \in\ell^2(\Z;\C^3)$, where only the second element of each 3-block of $y_{k_1}$ is kept. 
This second component is closely related to the modal representation of the vorticity $\rot u$, see the proof of Theorem \ref{th-u-decay} and Remark \ref{rem:isometry-rotu} below for details.
Analogously, the three matrices in \eqref{ODE:Oseen:y2} are obtained by keeping only the $(2,2)$-element of each submatrix, and they act on $\ell^2(\Z;\C)$. Hence, $\mE_{2,2}=\mI$. 
The infinite matrix
\begin{equation}\label{DAE:Oseen:staircase:R}
\widehat{\mR}_{2,2}
=\nu \diag(k_2^2;\ k_2\in\Z)
=\nu \diag(\ldots,\ 9,4,1,0,1,4,9,\ \ldots)
\end{equation}
is positive semi-definite but not coercive. 
Moreover it represents an unbounded self-adjoint operator on $\ell^2(\Z;\C)$ with dense domain 
\begin{equation} \label{domain:R_22}
  \mathcal D(\widehat{\mR}_{2,2})=\{z\in \ell^2(\Z;\C)\, |\, \sum_{j\in\Z} j^4\,|z(j)|^2<\infty\}\ . 
\end{equation}
Furthermore, 
\begin{multline} \label{VhRzz} 
  (\widehat{\mR}_{2,2})^{1/2}=\sqrt\nu \diag(|k_2|;\ k_2\in\Z)\ ,
\\
  \cD\big((\widehat{\mR}_{2,2})^{1/2}\big) =
  \{z\in \ell^2(\Z;\C)\, |\, \sum_{j\in\Z} j^2\,|z(j)|^2<\infty\}\ .
\end{multline}
The operator $(\widehat{\mJ}_{k_1})_{2,2} \in\cL(\ell^2(\Z;\C))$ can be represented as the skew-adjoint tridiagonal matrix
\begin{equation}\label{DAE:Oseen:staircase:J}
\begin{split}
&\big(\widehat{\mJ}_{k_1}\big)_{2,2}
\\
&=
{\scriptsize 
\tfrac{k_1}2 \!
\begin{bmatrix}
\ddots & \ddots & \ddots & & & &
\\
 & -\frac{k_1^2 +2}{\sqrt{k_1^2+1}\sqrt{k_1^2+4}} & 0 & \frac{k_1^2}{\sqrt{k_1^2+1}\ |k_1|} & & &
\\
 & & -\frac{k_1^2}{|k_1| \sqrt{k_1^2+1}} & 0 & \frac{k_1^2}{|k_1|\ \sqrt{k_1^2+1}} & &
\\
 & & & -\frac{k_1^2}{\sqrt{k_1^2+1}\ |k_1|} & 0 & \frac{k_1^2+2}{\sqrt{k_1^2+1}\ \sqrt{k_1^2+4}} &
\\ 
 & & & & \ddots & \ddots & \ddots
\end{bmatrix}
\begin{array}{l}
 \\ k_2=-1 \\[1em] k_2=0 \\[1em] k_2=1 \\[1em]
\end{array} }
\\
& {\scriptstyle\hspace{7em} \widetilde k_2=-2 \hspace{3em} \widetilde k_2=-1 \hspace{3.25em} \widetilde k_2=0 \hspace{3.5em} \widetilde k_2=1 \hspace{3.5em} \widetilde k_2=2}
\end{split}\! .
\end{equation}
Its off-diagonal terms have the general form
\[
\big(\widehat{\mJ}_{k_1}\big)_{2,2} (k_2,k_2\pm 1)
= 
\pm \frac{k_1}2 \frac{|k|^2 \pm k_2}{|k|\ |k\pm e_2|} ,
\]
hence, for each fixed $k_1\ne0$, $\big(\widehat{\mJ}_{k_1}\big)_{2,2}$ is a bounded operator on $\ell^2(\Z;\C)$. 
Moreover, one easily sees that 
\[
  \big(\widehat{\mJ}_{k_1}\big)_{2,2} \big(\mathcal D(\widehat{\mR}_{2,2})\big)\subset \mathcal D(\widehat{\mR}_{2,2}) \ ,\quad 
  \big(\widehat{\mJ}_{k_1}\big)_{2,2} \big(\cD\big((\widehat{\mR}_{2,2})^{1/2}\big)\big)\subset \cD\big((\widehat{\mR}_{2,2})^{1/2}\big) \ .
\]
Thus, the operators $\big(\widehat{\mJ}_{k_1}\big)_{2,2}$ and $\widehat{\mR}_{2,2}$ satisfy the assumptions~\ref{prop:R}, \ref{prop:R1/2} in the Definition \ref{def:unb-op:HC-index} of the HC-index.
Hence, for fixed $k_1\ne 0$, $(\widehat{\mJ}_{k_1}-\widehat{\mR})_{2,2}$ in~\eqref{ODE:Oseen:y2} generates an analytic semigroup of contractions on~$\ell^2(\Z;\C)$. Moreover, 
either of the three conditions from Lemma \ref{lem:Op-Equivalence} can be used to determine the HC-index (see Proposition \ref{prop:Oseen:y2} below).\\

%
%
Altogether, due to~\cite[Theorem III.2.10]{EnNa00}, the operator $(\widehat{\mJ}_{k_1}-\widehat{\mR})_{2,2}$ with domain~$\cD(\widehat{\mR}_{2,2})$ generates an analytic semigroup on~$\ell^2(\Z;\C)$. 
\medskip

\noindent
{\bf Exponential decay of the dynamical part:}
Using DAE concepts we separated the dynamical part from the algebraic constraint. The following proposition studies the hypocoercivity of the dynamical part. More precisely,  
with the above setup we can now establish the exponential decay of $\big(y_{k_1}\big)_2\,$, i.e., the modal decomposition of the vorticity $\rot u$ (in fact in $\dot H^{-1}_{per}(\T^2)$, see Remark \ref{rem:isometry-rotu} for details):
\begin{proposition}\label{prop:Oseen:y2}
Let $b_1=\sin (x_2)$ in~\eqref{model:Oseen:anisotropic_2}.
Then, for each $k_1\in\Zo$, the modal dynamics~\eqref{ODE:Oseen:y2} is hypocoercive in the sense of~\eqref{exp-decay} in~$\ell^2(\Z;\C)$.
Moreover, its HC-index $\mHC=1$.
\end{proposition}
We remark that HC-index $\mHC=1$ was already illustrated on one example in the estimate \eqref{prop-norm-est}.
In the following, we will use the canonical unit vectors $(e_j)_{j\in\Z}$ as orthonormal basis for
~$\ell^2(\Z)$
, which are defined for $j,k\in\Z$ as
\begin{equation} \label{ONB:l2}
 e_j^k
 =
 \begin{cases}
  1 & \text{for } k=j, \\
  0 & \text{for } k\ne j.
 \end{cases}
\end{equation}
\begin{proof}

To determine the HC-index of~$\big(\hJ_{k_1} -\hR\big)_{2,2}$ in~\eqref{ODE:Oseen:y2}, we use Condition \ref{B:G*_surjective} in Lemma \ref{lem:Op-Equivalence}:
The HC-index of $\big(\hJ_{k_1} -\hR\big)_{2,2}$ cannot be zero, since the selfadjoint part $\hRzz$ has a nontrivial kernel $\ker(\hRzz) =\spn\{e_0\}$.

We prove that its HC-index is $1$ by verifying Condition~\ref{B:G*_surjective} with $m=1$ in Lemma~\ref{lem:Op-Equivalence}:
For any $h\in\cH$, we determine $v,w\in\cD\big(\VhRzz\big)$ such that 
\begin{equation} \label{h=Rv+JRw}
 h =\VhRzz v +\hJkzz\VhRzz w \ .
\end{equation}
To do this, we compute
\begin{align*}
\big(\widehat{\mJ}_{k_1}\big)_{2,2}\
\big(\widehat{\mR}_{2,2}\big)^{1/2} 
=
-\sqrt{\nu} \frac{k_1}2
&\begin{bmatrix}
\ddots & \ddots & -\frac{k_1^2+2}{|k|\ |k-e_2|} & & & &
\\
 & \ddots & 0 & 0 & & &
\\
 & 0 & \frac{k_1^2}{|k|\ |k+e_2|} & 0 & -\frac{k_1^2}{|k|\ |k-e_2|} & 0 &
\\
 & & & 0 & 0 & \ddots &
\\ 
 & & & & \frac{k_1^2+2}{|k|\ |k+e_2|} & \ddots & \ddots
\end{bmatrix}
\begin{array}{l}
 \\[-0.4em] {\scriptstyle k_2=-1} \\[0.4em] {\scriptstyle k_2=0} \\[0.4em] {\scriptstyle k_2=1} \\[1em]
\end{array} ,
\\ 
& {\scriptstyle\hspace{4.6em} 
\widetilde{k}_2=-1 \hspace{1em} \widetilde{k}_2=0 \hspace{1.3em} \widetilde{k}_2=1 }
\end{align*}
see also~\eqref{hJzz} below.
Since $\big(\VhRzz v\big)_0 =0$, we observe that
\begin{align*}
 h_0
&
=-\sqrt{\nu} \frac{k_1}2 \frac{k_1^2}{\sqrt{k_1^2+1}\ |k_1|} (w_{-1} -w_1) \ ,
\intertext{and choosing $w_{-1}=2w_1$ yields,}
 h_0
&=\underbracket{-\sqrt{\nu} \frac{k_1}2 \frac{k_1^2}{\sqrt{k_1^2+1}\ |k_1|}}_{=:\alpha} w_1 \ ,
\end{align*}
such that $w_1 =\alpha^{-1} h_0$.
In fact, defining $w$ as
\[
 w_j :=
\begin{cases}
 \alpha^{-1} h_0 &\text{if $j=1$}\ , \\
 2\alpha^{-1} h_0 &\text{if $j=-1$}\ , \\
 0 &\text{else}\ ,
\end{cases} 
\]
yields $w\in\cD\big(\VhRzz\big)$.
Then $h-\hJkzz\VhRzz w\in\range\VhRzz$ such that 
\[
 v
:= 
\Big(\VhRzz\Big)^+
\big(h-\hJkzz\VhRzz w\big)
\in\cD\big(\VhRzz\big) \ ,
\]
where $\Big(\VhRzz\Big)^+=\frac1{\sqrt\nu} \diag\big(...,\frac13,\frac12,1,0,1,\frac12,\frac13,...\big)$ denotes the Moore-Penrose inverse.
Thus, for all $h\in\cH$, there exist $v,w\in\cD\big(\VhRzz\big)$ such that~\eqref{h=Rv+JRw} holds. This proves Condition~\ref{B:G*_surjective} with $m=1$ in Lemma~\ref{lem:Op-Equivalence} and finishes the proof of the statement about the HC-index.

We remark that determining the hypocoercivity index of 
$-{\mA}_{k_1}:=(\hR-\hJ_{k_1})_{2,2}$ directly via the coercivity condition in Definition~\ref{def:unb-op:HC-index} would be much more tedious, since both the matrix $-{\mA}_{k_1}$ and the corresponding left-hand-side of \eqref{Op:J-R:unbdd} (with $m=1$) are 5-diagonal, see Appendix~\ref{app:QuantitativeEstimate}.

\bigskip
In order to derive a hypocoercivity estimate~\eqref{exp-decay} for the evolution equation \eqref{ODE:Oseen:y2} we  proceed similarly to~\S4.3 in~\cite{AAC16} (there for a linear transport-reaction equation in 1D) and construct a strict Lyapunov functional~$\|y\|_{\mX_{k_1}}^2 := \ip{y}{\mX_{k_1} y}_{\ell^2(\Z)}$ for each $k_1\in\Zo$.
In the following, we will use the canonical unit vectors $(e_j)_{j\in\Z}$ as orthonormal basis for
~$\ell^2(\Z)$
, which are defined for $j,k\in\Z$ as
\[
 e_j^k
 =
 \begin{cases}
  1 & \text{for } k=j, \\
  0 & \text{for } k\ne j.
 \end{cases}
\]
We construct (an ansatz for) a strict Lyapunov functional using~\cite[Algorithm 3]{AAM21}, see Appendix~\ref{app:algorithm:1Pi}, starting with $\Pi_0=\mI$, $\tA_0=-\big(\widehat{\mJ}_{k_1}\big)_{2,2}$, and $\tB_0=\widehat{\mR}_{2,2}$.
The orthogonal projection~$\tPi_1$ onto  $\ker(\tB_0\tB_0)$ is given as $\tPi_1=e_0 e_0^\top$, such that $\Pi_1:=\tPi_1 =e_0 e_0^\top$. 
Here, by abuse of notation, sequences (such as $e_0$) are interpreted as infinite column vectors and $e_0 e_0^\top$ is the outer product of two (infinite) vectors such that $e_0 e_0^\top \in\cB(\ell^2(\Z))$ with 
\[
 \big(e_0 e_0^\top\big)_{j,k\in\Z}
=\begin{cases}
  1 & \text{for } k=j=0\ , \\
  0 & \text{else } \ .
 \end{cases}
\]
Then,
\[
 \tA_1
:= \Pi_1\tA_0\Pi_1 
 = -e_0 \underbracket{e_0^\top \big(\widehat{\mJ}_{k_1}\big)_{2,2} e_0}_{=0} e_0^\top
\]
and
\begin{align*}
 \tB_1
&:=
 \Pi_1 \tA_0 (\mI-\Pi_1)
 =
 \tfrac{k_1}2
 \begin{bmatrix}
 \ddots & \ddots & \ddots & & & &
 \\
 & 0 & 0 & 0 & & &
 \\
 & & \frac{k_1^2}{|k_1| \sqrt{k_1^2+1}} & 0 & -\frac{k_1^2}{|k_1|\ \sqrt{k_1^2+1}} & &
 \\
 & & & 0 & 0 & 0 &
 \\ 
 & & & & \ddots & \ddots & \ddots
 \end{bmatrix}
{\scriptsize  
 \begin{array}{l}
 \\ k_2=-1 \\[0.5em] k_2=0 \\[0.5em] k_2=1 \\[1em]
 \end{array} } , 
\\
&=
 \tilde{\alpha}
 \begin{bmatrix}
 \ddots & \ddots & \ddots & & & &
 \\
 & 0 & 0 & 0 & & &
 \\
 & & 1 & 0 & -1 & &
 \\
 & & & 0 & 0 & 0 &
 \\ 
 & & & & \ddots & \ddots & \ddots
 \end{bmatrix}
 =
 \tilde{\alpha} e_0 \underbracket{[\ldots,0,1,0,-1,0,\ldots]}_{=(e_{-1} -e_1)^\top} ,
\end{align*}
where $\tilde{\alpha} =\tfrac{k_1}2 \frac{k_1^2}{|k_1| \sqrt{k_1^2+1}}$.
In the next iteration,
\begin{align*}
 \tB_1 \tB_1^H
&=
 \tilde{\alpha}^2 e_0 (e_{-1} -e_1)^\top \big(e_0 (e_{-1} -e_1)^\top\big)^H
\\
&=
 \tilde{\alpha}^2 e_0 (e_{-1} -e_1)^\top \; (e_{-1} -e_1) e_0^\top 
\\
&=
 2 \tilde{\alpha}^2 (e_0 e_0^\top) .
\end{align*}
Hence, the orthogonal projection~$\tPi_2$ onto $\ker \big(\tB_1\tB_1^H\big)$ is given as $\tPi_2 =\mI -e_0 e_0^\top =\mI -\Pi_1$.
Due to $\Pi_2 :=\tPi_2 \Pi_1 =0$, the iteration terminates after one iteration in agreement with $\mHC=1$.
Finally, the ansatz for the weight matrix $\mX_{k_1}$ is chosen as 
\begin{equation}\label{model:Oseen:X}
\mX_{k_1}
:= \Pi_0+\epsilon_1 (\tA_0\Pi_1 + \Pi_1 \tA_0^H) =
\mI +\underbracket{\epsilon_1 \frac{k_1^3}{2\sqrt{k_1^2+1}\ |k_1|}}_{=:\epsilon_{k_1}\in\R}
\underbracket{\left[ \begin{array}{c|ccc|c}
 & & & & \\
\hline
 & 0 &-1 & 0 & \\
 &-1 & 0 & 1 & \\
 & 0 & 1 & 0 & \\
\hline
 & & & & 
\end{array}\right]}_{=:\mY} \ ,
\end{equation}
for some $\epsilon_1>0$ to be determined. We remark that all blank elements in the matrix in \eqref{model:Oseen:X} are zero. 
In the finite-dimensional setting, for sufficiently small $\epsilon_1>0$, the squared weighted norm $\|\cdot\|_{\mX_{k_1}}^2$ yields a strict Lyapunov functional, see~\cite{AAC22b}.
In this infinite dimensional example, we shall verify this statement directly.
The infinite matrix~$\mX_{k_1}$ is positive definite for $|\epsilon_{k_1}|<1/{\sqrt2}$.
Moreover,
\begin{equation}\label{DAE:Oseen:X}
 (1-\sqrt{2} |\epsilon_{k_1}|)\mI
 \leq
 \mX_{k_1}
 \leq 
 (1+\sqrt{2} |\epsilon_{k_1}|)\mI .
\end{equation}

\medskip
Finally, the coefficients $\epsilon_{k_1},\, k_1\ne0$ 
should be chosen such that the Lyapunov matrix inequalities (LMIs) 
\begin{equation}\label{DAE:Oseen:LMI}
 \mA_{k_1}^H \mX_{k_1} +\mX_{k_1}\mA_{k_1} +2\mu_{k_1}\mX_{k_1}
 \leq
 0\,,\quad k_1\ne0
\end{equation}
hold with $\mu_{k_1}$ as large as possible.
Due to Lemma~\ref{lemma:Oseen:LMI} (below), there exists an $\alpha_{\min}>0$ such that for all $\alpha\in(0,\alpha_{\min})$, the matrices $\mX_{k_1}$ defined in~\eqref{model:Oseen:X} with $\epsilon_{k_1}:=\alpha/{k_1}$ satisfy the LMIs~\eqref{DAE:Oseen:LMI}, where the decay rates~$\mu_{k_1}\ge\lambda_{1,\min}/4>0$ are uniformly bounded from below for all $k_1\in\Zo$; the bound $\lambda_{1,\min}=\lambda_{1,\min}(\alpha)$ is defined in~\eqref{model:Oseen:lambda1}. 
Consequently, for $\alpha\in(0,\alpha_{\min})$, the squared weighted norms $\|(y_{k_1})_2\|_{\mX_{k_1}}^2,\, k_1\in\Zo$
are strict Lyapunov functionals for the dynamics of each $(y_{k_1})_2$. More precisely, 
\begin{multline}\label{model:Oseen:decay:y2}
 \|(y_{k_1})_2(t)\|_{\ell^2(\Z)}
 \leq
 \sqrt{\kappa(\mX_{k_1})} e^{-t \lambda_{1,\min}/4} \|(y_{k_1})_2(0)\|_{\ell^2(\Z)}
\\ 
 \leq
 \sqrt{\tfrac{1+\sqrt{2}\ \alpha}{1-\sqrt{2}\ \alpha}} e^{-t \lambda_{1,\min}/4} \|y_{k_1,2}(0)\|_{\ell^2(\Z)} \ ,\quad t \geq 0\,,
\end{multline}
where $\alpha<\alpha_{\min}\leq 1/\sqrt 2$ and $\kappa(\mX_{k_1}) =\|\mX_{k_1}\|_2\ \|\mX_{k_1}^{-1}\|_2 = \frac{1+\sqrt2 \alpha/|k_1|}{1-\sqrt2 \alpha/|k_1|}$ is the condition number of~$\mX_{k_1}$.
This finishes the proof (of the first statement) of Proposition~\ref{prop:Oseen:y2}.
\end{proof}
The following lemma was used in the above proof. 
There, we did not intend to find the optimal decay rate of each system in~\eqref{ODE:Oseen:y2} (and similarly of~\eqref{DAE:Oseen:staircase}). Hence, we shall (only) prove that there exists an $\epsilon_1>0$ such that $\mQ_{k_1} :=-(\mA_{k_1}^H \mX_{k_1} +\mX_{k_1}\mA_{k_1})$ is positive definite; afterwards we shall determine $\mu_{k_1}>0$ such that~\eqref{DAE:Oseen:LMI} holds.
\begin{lemma}\label{lemma:Oseen:LMI}
Consider $\mA_{k_1} :=(\widehat{\mJ}_{k_1})_{2,2} -\widehat{\mR}_{2,2}$ with the matrices defined in~\eqref{DAE:Oseen:staircase:R}--\eqref{DAE:Oseen:staircase:J}.
For any fixed $\nu>0$, there exists $\alpha_{\min}>0$ (defined in~\eqref{cond:Oseen:alpha}) such that for all $\alpha\in(0,\alpha_{\min})$, the matrices $\mX_{k_1}$ defined in~\eqref{model:Oseen:X} with $\epsilon_{k_1}:=\alpha/k_1$ satisfy the Lyapunov matrix inequality~\eqref{DAE:Oseen:LMI}, where $\mu_{k_1}>0$ is uniformly bounded from below by $\lambda_{1,\min}/4>0$ (defined in~\eqref{model:Oseen:lambda1}) for all $k_1\in\Zo$.
\end{lemma}
The technical proof of this lemma is deferred to Appendix~\ref{appendix_three}.\\

\noindent
{\bf Exponential decay of the full system:
}Combining the modal decay of Proposition \ref{prop:Oseen:y2} with the decay of the $(0,k_2)$-modes we shall obtain next the hypocoercivity estimate of the anisotropic Oseen equation \eqref{model:Oseen:anisotropic_2} in 
state domain.
\begin{align*}
  u(t)
&\in \cH := H_{per}(\diver 0, \T^2)\\
  u(t)-u_\infty 
&\in \wcH:=\{ u\in\cH\ |\ \int_{\T^2} u \dd[x]=0 \}\ ,  
\end{align*}
where the steady state $u_\infty:=\frac{1}{4\pi^2}\int_{\T^2} u(x,0)\dd[x]$ equals the constant-in-$t$ mode $\phi_0$. 
Moreover, both $\cH$ and  $\wcH$ are closed subspaces of $(L^2(\T^2))^2$. For the Fourier decomposition \eqref{Fourier} we have in $\wcH$
\begin{equation}\label{parseval}
  \|u(t)-u_\infty\|_{\wcH}^2
  = \|u(t)-u_\infty\|_{(L^2(\T^2))^2}^2
  =\frac{1}{4\pi^2} \sum_{k\ne0} |\phi_k(t)|^2\,.
\end{equation}
\begin{theorem}\label{th-u-decay}
Let $b_1=\sin (x_2)$ in~\eqref{model:Oseen:anisotropic_2} and $u(0)\in H_{per}(\diver 0, \T^2)$.
Then
\begin{multline}\label{model:Oseen:L2-decay}
 \|\nabla p(t)\|_{(L^2(\T^2))^2}\le \|u(t)-u_\infty\|_{(L^2(\T^2))^2}
\\ 
 \leq 
 \sqrt{\tfrac{1+\sqrt{2}\ \alpha}{1-\sqrt{2}\ \alpha}} e^{-t \min(\nu,\,\lambda_{1,\min}/4)} \|u(0)-u_\infty\|_{(L^2(\T^2))^2} \ ,\quad t \geq 0\,,
\end{multline}
where $\alpha$ and $\lambda_{1,\min}$ are as in \eqref{model:Oseen:decay:y2}.
\end{theorem}
\begin{proof}
We  first prove that
\begin{equation}\label{y-norm}
  \|\big(y_{k_1}\big)_2\|_{\ell^2(\Z)} = 
  \|\phi_{(k_1,\cdot)}\|_{\ell^2(\Z;\C^2)}\,,\quad \mbox{for all}\,k_1\ne0\,
\end{equation} 
for divergence-free flow fields $u\in\wcH$, and that 
\begin{equation}\label{y-form}
  \big(y_{k_1}\big)_2=\big\{\frac{1}{|k|}\big(-k_2\ukI+k_1\ukII\big)\big\}_{k_2\in\Z}\,,\quad k_1\ne0\,
\end{equation} 
is a bijection from the divergence-free subspace of $\ell^2(\Z;\C^2)$, i.e., for $ik\cdot \phi_k=0$, to $\ell^2(\Z)$; see also the remark on the leading $2\times2$-subblock of $\mP_{k_1}$ after \eqref{model:Oseen:Pk1}. 

To this end we first note that the right hand side of \eqref{y-form} is 
related to the vorticity of $u$, since the modal representation of $\rot u$ (in 2D) is $i(-k_2\ukI+k_1\ukII)$.

Given a divergence-free $\phi_{(k_1,\cdot)} \in \ell^2(\Z;\C^2)$ with $k_1\ne0$ we compute with \eqref{y-form}:
\begin{align*}
& \|\big(y_{k_1}\big)_2\|_{\ell^2(\Z)}^2 
\\
&= \sum_{k_2\in\Z} \frac{1}{|k|^2} \big(k_2^2|\ukI|^2 -k_1k_2 \overline{\ukI} \, \ukII -k_1k_2 \ukI \, \overline{\ukII}+k_1^2|\ukII|^2\big) 
\\
&= \sum_{k_2\in\Z} \frac{1}{|k|^2} \big(k_2^2|\ukI|^2 +k_1^2|\ukI|^2+k_2^2|\ukII|^2 +k_1^2|\ukII|^2\big) =
  \|\phi_{(k_1,\cdot)}\|_{\ell^2(\Z;\C^2)}^2\,,
\end{align*}
where we used $k\cdot \phi_k=0$ twice.

For the other direction, given a $\big(y_{k_1}\big)_2 \in \ell^2(\Z)$ with $k_1\ne0$ we define the divergence-free flow field
\[
  \ukI:=-\frac{k_2}{|k|} \big(y_{k_1}\big)_2\,,\quad 
  \ukII:=\frac{k_1}{|k|} \big(y_{k_1}\big)_2\,,\quad   
  k_2\in\Z\,,
\]
which is compatible with \eqref{y-form}. Then we have
\[
 \|\phi_{(k_1,\cdot)}\|_{\ell^2(\Z;\C^2)}^2
= \sum_{k_2\in\Z} \Big(\frac{k_2^2}{|k|^2} + \frac{k_1^2}{|k|^2} \Big)|\big(y_{k_1}\big)_2|^2 
= \|\big(y_{k_1}\big)_2\|_{\ell^2(\Z)}^2\,,
\]
and this proves the isometry and bijectivity.

Finally, we sum up the (square of the) modal inequalities \eqref{model:Oseen:decay:y2} for $k_1\ne0$ and the inequalities \eqref{k10-decay} for $k_1=0$ but $k_2\ne0$. This yields the claimed decay estimate \eqref{model:Oseen:L2-decay} for $u(t)-u_\infty$ with using \eqref{parseval}. The decay of $\|\nabla p(t)\|_{(L^2(\T^2))^2}$ then follows from the estimate \eqref{p-u-estimate}.
\end{proof}

\begin{remark}\label{rem:isometry-rotu}
The modal isometry \eqref{y-norm} can be extended to physical space by combining all modes $\big(y_{k_1}\big)_2$. 
Recalling from \eqref{y-form} the dependence of $y_2=y_2[u]$ on $u$, let $y_2:=\big\{\big(y_{k}\big)_2\big\}_{k\in\Z^2\setminus \{0\}}$. For scalar functions $f$ on $\T^2$ we define the following homogeneous Sobolev space via the Fourier decomposition of $f$:
\[
  f\in \dot H_{per}^{-1}(\T^2) \quad :\Leftrightarrow \quad \|f\|_{\dot H_{per}^{-1}(\T^2)}^2 :=
  \frac1{4\pi^2} \sum_{k\ne0} \frac1{|k|^2}|f_k|^2 <\infty\,.
\]

Then we have the following relation for divergence-free flow fields $u$ on $\T^2$ with vanishing average, i.e.\ $\int_{\T^2} u(x)\dd[x]=0$: 
The space $\dot H_{per}^{-1}(\T^2)$\; [for $\rot u=-\partial_{x_2}\uI+\partial_{x_1}\uII$] is isometrically isomorphic to 
\[\wcH =\{ u\in H_{per}(\diver \,0,\T^2)\ |\ \int_{\T^2} u \dd[x]=0 \}\;\]
[for $u$]. 
More precisely we have
\begin{multline*}
\frac1{4\pi^2} \|y_2[u]\|^2_{\ell^2(\Z^2\setminus\{0\})}
 = \frac1{4\pi^2} \sum_{k\ne0} |(y_k)_2|^2
 = \frac1{4\pi^2} \sum_{k\ne0} \Big|\frac1{|k|}\big(-k_2\ukI+k_1\ukII\big)\Big|^2 
\\
 = \|\rot u\|^2_{\dot H_{per}^{-1}(\T^2)}
 = \frac1{4\pi^2} \|\phi_k\|^2_{(\ell^2(\Z^2\setminus\{0\}))^2}
 = \|u\|^2_{(L^2(\T^2))^2}\,,
\end{multline*}
where we used \eqref{y-norm}. This relation shows that $\|u(t)\|_{(L^2(\T^2))^2}$ describes all combined modes of the dynamical part in \eqref{ODE:Oseen:y2} by summing over $k_1\in\Z$.
\end{remark}

\medskip
In Proposition \ref{prop:Oseen:y2} we established the exponential decay of the dynamical mode component $(y_{k_1})_2(t)$. 
Next we shall extend this result by including the enslaved components $(y_{k_1})_1\equiv0$ and $(y_{k_1})_3(t)=i\,p_{k_1}(t)$.
\begin{proposition}\label{prop:Oseen:y}
Let $b_1=\sin (x_2)$ in~\eqref{model:Oseen:anisotropic_2}.
Then, for each $k_1\in\Zo$, the modal dynamics~\eqref{DAE:Oseen:staircase} is hypocoercive in the sense of~\eqref{exp-decay} in~$\ell^2(\Z;\C^3)$.
\end{proposition}
\begin{proof}
Noting again that the third block in the staircase form~\eqref{DAE:Oseen:staircase} is void, it follows that 
\begin{align*}
 y_{k_1} 
&=
 \mP_{k_1} w_{k_1}
 =
 \big[ \mP_{(k_1,k_2)} [\phi_{(k_1,k_2),1}, \phi_{(k_1,k_2),2}, p_{(k_1,k_2)}]^\top ;\ k_2\in\Z \big]
\\ 
&=:
 \big[ [y_{(k_1,k_2),1}, y_{(k_1,k_2),2}, y_{(k_1,k_2),3}]^\top ;\ k_2\in\Z \big]
 \in \ell^2(\Z;\C^3).
\end{align*}
In our infinite dimensional model, we deduce (in analogy to the results of~\cite[Corollary 1]{AAM21}) that
\begin{equation}\label{DAE:Oseen:y}
\begin{split}
 \big(y_{k_1}\big)_1 &= 0 \ , 
\\
 \big(y_{k_1}\big)_2 &\text{ is governed by~\eqref{ODE:Oseen:y2}} \ , 
\\ 
 \big(y_{k_1}\big)_3 
 &=
 (\widecheck{\mJ}_{k_1})_{3,1}^{-H} (-(\widecheck{\mJ}_{k_1})_{2,1}^H -\mR_{2,1}^H) \big(y_{k_1}\big)_2 \ ,
\end{split}
\end{equation}
where $(\widecheck{\mJ}_{k_1})_{3,1} =\diag(|k|=\sqrt{k_1^2 +k_2^2}; \ k_2\in\Z)$, $\mR_{2,1}=0$, and $(\widecheck{\mJ}_{k_1})_{2,1}$ is a symmetric tridiagonal matrix whose diagonal elements are zero and the off-diagonal elements are $(\widecheck{\mJ}_{k_1})_{2,1}(k_2,k_2\pm 1) =k_1^2 /(2\ |k|\ |k\pm e_2|)$ with $e_2:=[0,\, 1]^\top$.
Then, for fixed $k_1\ne 0$, the third component of \eqref{DAE:Oseen:y} reads explicitly
\begin{equation}\label{y2-y3-relation}    
  \big(y_{k_1}\big)_3=
  \big(y_{k}\big)_3= -\frac{k_1^2}{2|k|^2}\left(\frac{\big(y_{k-e_2}\big)_2}{|k-e_2|}+\frac{\big(y_{k+e_2}\big)_2}{|k+e_2|}\right) , \quad k=[k_1,\,k_2]^\top,\quad k_2\in\Z.
\end{equation}
Since $(y_{k_1})_3$ is a linear combination of  $(y_{k_1})_2$ with $k_2$-uniformly bounded coefficients, solutions $y_{k_1}(t)$ for consistent initial data will converge to $0$ with uniform exponential rate $\lambda_{1,\min}/4>0$ for all $k_1\in\Zo$. 
Since the multiplying factors in the modal relation \eqref{y2-y3-relation} are $\mathcal O(1/|k|)$ and by using \eqref{y-norm}, \eqref{parseval}, we see that the decay of $p(t)$ is in a Sobolev space one level higher than the decay of $u(t)-u_\infty$. With different tools, this was already reflected in \eqref{model:Oseen:L2-decay}. 

A consistent initial value $y_{k_1}(0)$ again has to satisfy~\eqref{DAE:Oseen:y} such that 
\begin{equation}\label{DAE:Oseen:y0}
\begin{split}
 \big(y_{k_1}\big)_1(0) &= 0 \ , 
\\
 \big(y_{k_1}\big)_2(0) &\in\ell^2(\Z) \ , 
\\ 
 \big(y_{k_1}\big)_3(0) 
 &=
 (\widecheck{\mJ}_{k_1})_{3,1}^{-H} (-(\widecheck{\mJ}_{k_1})_{2,1}^H -\mR_{2,1}^H) \big(y_{k_1}\big)_2(0) \ .
\end{split}
\end{equation}
Then, the associated solution $y_{k_1}(t)$ of~\eqref{DAE:Oseen:staircase} satisfies
\begin{equation} \label{model:Oseen:decay:y}
\begin{split}
 \| y_{k_1}(t)\|^2
&=
 \| \big(y_{k_1}\big)_2(t)\|^2 +\| \big(y_{k_1}\big)_3(t)\|^2
\\
&=
 \| \big(y_{k_1}\big)_2(t)\|^2 +\| (\widecheck{\mJ}_{k_1})_{3,1}^{-1} (\widecheck{\mJ}_{k_1})_{2,1}\, \big(y_{k_1}\big)_2(t)\|^2 
\\
&\leq 
 \Big(1 +\big\|(\widecheck{\mJ}_{k_1})_{3,1}^{-1}  (\widecheck{\mJ}_{k_1})_{2,1}\big\|^2 \Big) \| \big(y_{k_1}\big)_2(t)\|^2 
\\
&\leq 
 \Big(1 +\big\|(\widecheck{\mJ}_{k_1})_{3,1}^{-1}  (\widecheck{\mJ}_{k_1})_{2,1}\big\|^2 \Big) {\tfrac{1+\sqrt{2}\ \alpha}{1-\sqrt{2}\ \alpha}} \| \big(y_{k_1}\big)_2(0)\|^2 e^{-\lambda_{1,\min} t/2}
\\
&\leq 
 \Big(1 +\big\|(\widecheck{\mJ}_{k_1})_{3,1}^{-1}  (\widecheck{\mJ}_{k_1})_{2,1}\big\|^2 \Big) {\tfrac{1+\sqrt{2}\ \alpha}{1-\sqrt{2}\ \alpha}} \| y_{k_1}(0)\|^2 e^{-\lambda_{1,\min} t/2} \ ,
\end{split}
\end{equation}
due to~\eqref{model:Oseen:decay:y2}, where $\alpha<\alpha_{\min}\leq 1/\sqrt 2$, $\lambda_{1,\min}=\lambda_{1,\min}(\alpha)$ is defined in~\eqref{model:Oseen:lambda1}, and $\|\mT\|$ denotes the operator norm of a bounded operator~$\mT\in\mathcal B(\ell(\Z))$.
This finishes the proof of Proposition~\ref{prop:Oseen:y}.
\end{proof}
As a final step we note that the (inverse) staircase transformation also implies exponential decay of the variable $\xk :=[\ukI,\ukII,p_k]^\top$ which is the modal representation of $[u,\, p]^\top$:
\begin{proposition}\label{prop:Oseen:w}
Let $b_1=\sin (x_2)$ in~\eqref{model:Oseen:anisotropic_2}.
Then, for each $k_1\in\Z$, the modal dynamics~\eqref{DAE:Oseen} is hypocoercive in the sense of~\eqref{exp-decay} in~$\ell^2(\Z;\C^3)$.
\end{proposition}
\begin{proof}
For $k_1=0$, the (evolution equations of the) modes~$w_{k_1}$ are decoupled.
Following the analysis of the family of decoupled DAEs~\eqref{DAE:Stokes} (with $k_1=0$), and in particular \eqref{k10-decay} shows that $w_0$ converges to the infinite complex vector
\[
 w_{0}^\infty 
 := 
 [\ldots;\, 0, 0, 0;\, \phi_{(0,0),1}, \phi_{(0,0),2}, p_{(0,0)};\, 0, 0, 0;\, \ldots]^\top 
 \ ,
\]
(which corresponds to the constant equilibrium $(\phi_0,p_0)$) with the exponential decay rate $\nu$($=\min_{k_2\ne 0} (\nu k_2^2)$).

For $k_1\in\Zo$, consistent initial data~$w_{k_1}(0)$ of system~\eqref{DAE:Oseen} has the form $w_{k_1}(0)=\mP_{k_1}^H y_{k_1}(0)$ where $\mP_{k_1}$ is defined in~\eqref{model:Oseen:Pk1} and $y_{k_1}(0)$ is given as~\eqref{DAE:Oseen:y0}.
Then, the associated solution~$w_{k_1}(t)$ of system~\eqref{DAE:Oseen} satisfies
\begin{align*}
 \| w_{k_1}(t)\|^2 
&=
 \| \mP_{k_1}^H y_{k_1}(t)\|^2
 =
 \| y_{k_1}(t)\|^2 \\
&\leq 
 \Big(1 +\big\|(\widecheck{\mJ}_{k_1})_{3,1}^{-1}  (\widecheck{\mJ}_{k_1})_{2,1}\big\|^2 \Big)  {\tfrac{1+\sqrt{2}\ \alpha}{1-\sqrt{2}\ \alpha}} \| y_{k_1}(0)\|^2 e^{-\lambda_{1,\min} t/2} \\
&= 
 \Big(1 +\big\|(\widecheck{\mJ}_{k_1})_{3,1}^{-1}  (\widecheck{\mJ}_{k_1})_{2,1}\big\|^2 \Big)  {\tfrac{1+\sqrt{2}\ \alpha}{1-\sqrt{2}\ \alpha}} \| w_{k_1}(0)\|^2 e^{-\lambda_{1,\min} t/2} \ ,
\end{align*}
due to~\eqref{model:Oseen:decay:y}, where $\alpha<\alpha_{\min}\leq 1/\sqrt 2$, $\lambda_{1,\min}=\lambda_{1,\min}(\alpha)$ is defined in~\eqref{model:Oseen:lambda1}, and $\|\mT\|$ denotes the operator norm of a bounded operator~$\mT\in\mathcal B(\ell(\Z))$.

Altogether, for consistent initial data, solutions of system~\eqref{DAE:Oseen} (and resp.~\eqref{eq:Oseen:Fourier}) converge to the constant equilibrium with a uniform exponential rate.
\end{proof}


\section{Conclusions}\label{sec:conclusions}

After extending the notion of hypocoercivity index to evolution equations in (infinite dimensional) Hilbert spaces, we have performed the analysis of the long-time decay behavior of three variants of isotropic and anisotropic Oseen-type equations from fluid dynamics (for simplicity on a 2D torus). Due to the torus setting we used DAE theory in Fourier space to classify the hypocoercivity index. These equations are either coercive, hypocoercive with index 1, or even not hypocoercive (the latter showing exponential convergence only to a traveling wave solution).

\begin{appendices}
\section{Appendix}\label{sec:appendix}

\subsection{Hypocoercivity in linear semi-dissipative DAEs}
\label{app:DAEs:hypocoercivity}

Here, we recall the basic theory of hypocoercivity for finite-dimensional linear semi-dissipative Hamiltonian DAEs~\cite{AAM21}:

\begin{definition}[{\cite[Definition~4]{AAM21}}]\label{def:hypoDAE}
A matrix pencil $\lambda \mE-\mA$ is called~\emph{negative hypocoercive} if the pencil is regular, of DAE-index at most two and the finite eigenvalues of the pencil $\lambda \mE-\mA$ have negative real part.
\end{definition}
We note that a regular pencil might not have \emph{any} finite eigenvalues, in which case the last condition would be void.
Due to \cite[Theorem~3]{AAM21}, a linear semi-dissipative Hamiltonian DAE system~\eqref{DAE:EA} with a regular pencil~$\lambda \mE-\mA$ only has finite eigenvalues with non-positive real part. 

The definition of the hypocoercivity index for DAEs is based on a staircase form of DAEs, see \cite[Lemma 5]{AAM21}:
\begin{lemma}[Staircase form for triple $(\mE,\mJ,\mR)$] \label{lem:tSF}
Let $\mE,\mJ,\mR\in\Cnn$ satisfy $\mE=\mE^H\geq 0$, $\mR=\mR^H\geq 0$ and $\mJ=-\mJ^H$.
Then there exists a unitary matrix $\mP\in\Cnn$, such that $\widecheck \mE :=\mP\ \mE\ \mP^H$, $\widecheck \mJ :=\mP\ \mJ\ \mP^H$ and $\widecheck \mR :=\mP\ \mR\ \mP^H$ satisfy
\begin{equation}\label{staircase:EJR}
\widecheck \mE 
=:\begin{bmatrix}
\mE_{1,1} & \mE_{2,1}^H & 0 & 0 & 0 \\
\mE_{2,1} & \mE_{2,2} & 0 & 0 & 0 \\
0  & 0 & 0 & 0 & 0 \\
0  & 0 & 0 & 0 & 0 \\
0  & 0 & 0 & 0 & 0
\end{bmatrix}\! ,
\ 
\widecheck \mJ 
=:\begin{bmatrix}
\mJ_{1,1} & -\mJ_{2,1}^H & -\mJ_{3,1}^H & -\mJ_{4,1}^H & 0\\
\mJ_{2,1} & \mJ_{2,2} & -\mJ_{3,2}^H& 0 & 0\\
\mJ_{3,1} & \mJ_{3,2} & \mJ_{3,3}&  0 & 0\\
\mJ_{4,1} & 0& 0 & 0 &0\\
 0 & 0 & 0 & 0 & 0
\end{bmatrix}\! ,
\ 
\widecheck \mR 
=:\begin{bmatrix}
\mR_{1,1} & \mR_{2,1}^H & \mR_{3,1}^H& 0 & 0 \\
\mR_{2,1} & \mR_{2,2} & \mR_{3,2}^H &0 & 0  \\
\mR_{3,1} & \mR_{3,2} & \mR_{3,3} &0 & 0 \\
0  & 0 & 0 & 0 & 0 \\
0  & 0 & 0 & 0 & 0
\end{bmatrix}\! .
\end{equation}
These three matrices are partitioned in the same way, with (square) diagonal block matrices of sizes $n_1,n_2,n_3,n_4=n_1,n_5\in\N_0$.
If the block matrices $\mE_{1,1}$, $\mE_{2,2}$ (as well as $\mE_{2,1}$) are present, then the matrices $\mE_{1,1}$, $\mE_{2,2}$ (as well as $\begin{bmatrix} \mE_{1,1} & \mE_{2,1}^H \\ \mE_{2,1} & \mE_{2,2}  \end{bmatrix}$) are positive definite.
If the block matrices $\mJ_{4,1}$, $\mJ_{3,3} -\mR_{3,3}$ are present, then the matrices $\mJ_{4,1}$, $\mJ_{3,3}-\mR_{3,3}$ are invertible. 
\end{lemma}
The proof is given as a constructive algorithm, see~\cite[Algorithm 5]{AAM21}, which is reproduced here as Algorithm~\ref{algorithm:staircase:E_J-R} below.

\smallskip
\setcounter{algorithm}{4} 
\begin{breakablealgorithm}
\caption{Staircase Algorithm for triple~$(\mE,\mJ,\mR)$}
\label{algorithm:staircase:E_J-R}
\begin{algorithmic}[1]
\Statex ----------- \textit{Step 1} -----------
\State Perform a spectral decomposition of~$\mE$ such that
\[
\mE =\mU_E \begin{bmatrix} \tE_{1,1} & 0 \\ 0 & 0 \end{bmatrix} \mU_E^H,
\]
with $\mU_E\in\Cnn$ unitary, $\tE_{1,1}\in \C^{\tilde n_1, \tilde n_1}$ positive definite or $\tilde n_1=0$.
\State Set 
\begin{align*}
\mP 
&:= \mU^H_E , \qquad
\tJ
:= \mU^H_E\ \mJ\ \mU_E
=\begin{bmatrix}
 \tJ_{1,1} & -\tJ_{2,1}^H \\
 \tJ_{2,1} & \tJ_{2,2}
 \end{bmatrix} ,
\\ %
\tR
&:= \mU^H_E\ \mR\ \mU_E
=\begin{bmatrix}
 \tR_{1,1} & \tR_{2,1}^H \\
 \tR_{2,1} & \tR_{2,2} \end{bmatrix} , \qquad
\tE :=\mU^H_E\ \mE\ \mU_E .
\end{align*}
\Statex ----------- \textit{Step 2} -----------
\If{$\tilde n_1<n$}
\State Apply~\cite[Lemma 2]{AAM21} to $\tJ_{2,2}-\tR_{2,2}\in\C^{(n-\tilde n_1)\times (n-\tilde n_1)}$ such that
\[
\mP_{2,2}\ (\tJ_{2,2}-\tR_{2,2})\ \mP_{2,2}^H
=\begin{bmatrix} \tSigma_{2,2} & 0 \\ 0 & 0\end{bmatrix},
\]
with $\tSigma_{2,2}\in \mathbb C^{\tilde n_2,\tilde n_2}$ invertible or $\tilde n_2=0$.
\EndIf
\State Set
\[
\mP_2:=\begin{bmatrix} \mI& 0 \\ 0 & \mP_{2,2} \end{bmatrix} \in\Cnn,\qquad \mP:=\mP_2 \mP.
\]
\State Set $\tE:=\mP_2\ \tE\ \mP_2^H$,
\[
\def\arraystretch{1.4}
\tJ := \mP_2\ \tJ\ \mP_2^H
=: \left[ \begin{array}{c|cc}
 \tJ_{1,1} & -\tJ_{2,1}^H & -\tJ_{3,1}^H  \\
 \hline 
 \tJ_{2,1} & \tJ_{2,2} & 0 \\
 \tJ_{3,1}& 0 & 0
\end{array}\right],\qquad
\tR := \mP_2\ \tR\ \mP_2^H
=: \left[ \begin{array}{c|cc}
 \tR_{1,1} & \tR_{2,1}^H & 0 \\
 \hline 
 \tR_{2,1} & \tR_{2,2} & 0 \\
 0 & 0 & 0
\end{array}\right],\quad
\]
with $\tJ_{2,2}-\tR_{2,2}=\tSigma_{2,2}$.
(The lines indicate the partitioning of the block matrices~$\tJ$ and~$\tR$ in the previous step.
Note that the positive semi-definiteness of the Hermitian matrix~$\mR$ implies the~$0$ structure in~$\tR$.)
\Statex ----------- \textit{Step 3} -----------
\State Define $\tilde n_3 :=n -\tilde n_1 -\tilde n_2$.
\If{$\tilde n_3>0$}
\State Perform an SVD of $\tJ_{3,1}$ such that
\[
\tJ_{3,1} = \mU_{3,1} \begin{bmatrix} \tSigma_{3,1} & 0\\ 0 & 0 \end{bmatrix} \mV^H_{3,1} \in \C^{\tilde n_3 \times \tilde n_1}\,,
\]
\hspace{12pt} with~$\tSigma_{3,1}\in\R^{n_1 \times n_1}$ nonsingular diagonal or $n_1 =0$.
\EndIf
\State Set
\[
\mP_3 :=\begin{bmatrix}
 \mV_{3,1}^H & & \\
  & \mI & \\
  & & \mU_{3,1}^H
\end{bmatrix} \in\Cnn,\qquad
\mP:= \mP_3 \mP.
\]
\State Set $\widecheck \mE
:= \mP_3\ \tE\ \mP_3^H$, $\widecheck\mJ
:= \mP_3\ \tJ\ \mP_3^H$, $\widecheck\mR
:= \mP_3\ \tR\ \mP_3^H$ such that
\begin{align*}
\widecheck \mE
&=: \left[ \begin{array}{cc|c|cc}
 \mE_{1,1} &    \mE_{2,1}^H    &  0 &0 & 0 \\
 \mE_{2,1} & \mE_{2,2} &  0    &   0      & 0 \\
\hline
 0  & 0  & 0  & 0  & 0 \\
\hline
 0  & 0  & 0  & 0  & 0 \\
 0  & 0  & 0  & 0  & 0
\end{array}\right],
\\
\widecheck\mJ
&=:\left[ \begin{array}{cc|c|cc}
\mJ_{1,1}  &  -\mJ_{2,1}^H  &  -\mJ_{3,1}^H  &  -\mJ_{4,1}^H  &  0 \\
\mJ_{2,1}  &  \mJ_{2,2}  &  -\mJ_{3,2}^H  &  0  &  0 \\
\hline
\mJ_{3,1}  &  \mJ_{3,2}  &  \mJ_{3,3}  &  0  &  0 \\
\hline
\mJ_{4,1}  &  0  &  0  &  0  &  0 \\
0  &  0  &  0  &  0  &  0
\end{array}\right],
\qquad %
\widecheck\mR
=:\left[ \begin{array}{cc|c|cc}
\mR_{1,1}  &  \mR_{2,1}^H  &  \mR_{3,1}^H &  0  &  0 \\
\mR_{2,1}  &  \mR_{2,2}  &  \mR_{3,2}^H  &  0  &  0 \\
\hline
\mR_{3,1}  &  \mR_{3,2}  &  \mR_{3,3}    &  0  &  0 \\
\hline
0  &  0  &  0  &  0  &  0 \\
0  &  0  &  0  &  0  &  0
\end{array}\right],
\end{align*}
which are of the desired form with $n_2:=\tilde n_1 -n_1$, $n_3 :=\tilde n_2$, $n_4:=n_1$, $n_5:=\tilde n_3 -n_4$.
The matrices~$\mJ_{4,1}:=\tSigma_{3,1}$ and $\mJ_{3,3}-\mR_{3,3} =\tJ_{2,2}-\tR_{2,2} =\tSigma_{2,2}$ are invertible.
\end{algorithmic}
\end{breakablealgorithm}

%
The pencil~$\lambda \mE-(\mJ-\mR)$ is associated to the DAE~\eqref{DAE:EA} with~$\mA=\mJ-\mR$. Using the above lemma, it can be transformed into a DAE in staircase form,
\begin{equation}\label{DAE:EA:hat:0}
\widecheck\mE\dot y =(\widecheck\mJ -\widecheck\mR)y\ ,
\qquad\text{with } y:=\mP x\ .
\end{equation}
In analogy to \eqref{staircase:EJR}, the vector $y$ can be partitioned as $y=(y_1,...,y_5)^T$ with subvectors of the respective length $n_1,...,n_5$.
For systems~\eqref{DAE:EA:hat:0} with~$n_5=0$, the underlying implicit ODE systems are given by the system in~$y_2$  that are obtained by eliminating all other variables.
For example, if $n_2>0$ and $n_3\geq 0$, then this yields systems of the form
\begin{equation}\label{underlyingODE:y}
\mE_{2,2} \dot y_2
=\widehat \mA_{2,2} y_2
=(\widehat \mJ_{2,2}-\widehat \mR_{2,2}) y_2,
\end{equation}
with $\mE_{2,2}=\widehat\mE_{2,2}$ Hermitian positive definite and $\widehat \mA_{2,2}$ semi-dissipative.
Here
\begin{multline*}
 \widehat\mJ_{2,2} :=(\widehat\mA_{2,2})_S \ , \quad
 \widehat\mR_{2,2} :=-(\widehat\mA_{2,2})_H \ , 
\\
 \text{where } 
 \widehat\mA_{2,2} 
=\begin{cases} 
\widecheck\mA_{2,2} \ , & \text{if } n_3=0\ , \\
\widecheck\mA_{2,2} -\widecheck\mA_{2,3}\widecheck\mA_{3,3}^{-1}\widecheck\mA_{3,2} \ , & \text{if } n_3>0\ .
 \end{cases}
\end{multline*}

This staircase form now allows to define the hypocoercivity index also for DAEs:

\begin{definition}[{\cite[Definition 5]{AAM21}}] \label{def:DAE:mHC}
Consider a linear semi-dissipative Hamiltonian DAE system~\eqref{DAE:EA} with a regular pencil~$\lambda \mE-\mA$ and the unitarily congruent DAE~\eqref{DAE:EA:hat:0} in staircase form~\eqref{staircase:EJR}.
If the underlying implicit ODE~\eqref{underlyingODE:y} is missing (present) then system~\eqref{DAE:EA} is said to exhibit \emph{(non-)trivial dynamics}.
In case of non-trivial dynamics, the~\emph{HC-index~$m_{HC}$} of~$\lambda \mE-\mA$ is defined as the HC-index of the system matrix~$(\mE_{2,2}^{1/2})^{-1} \widehat{\mA}_{2,2} (\mE_{2,2}^{1/2})^{-1}$ of~\eqref{underlyingODE:y} (in the sense of Definition \ref{def:matrix:hypocoercive}), otherwise it is defined as~$0$.
\end{definition}
The following proposition 
states that the HC-index characterizes the short time behavior of its solution propagator, but restricted to the dynamical subspace.
\begin{proposition}[{\cite[Proposition 3]{AAM21}}] \label{prop:DAE+HC-decay}
Consider the semi-dissipative Hamiltonian DAE~\eqref{DAE:EA} with a regular, negative hypocoercive pencil~$\lambda\mE-\mA$, DAE-index at most two, non-trivial dynamics, and consistent initial condition~$x(0)$.
Then its (finite) HC-index is $\mHC\in\N_0$, if and only if
\[
\|S(t)\|_{\mE}
=
1 -c t^a +\bigO(t^{a+1})
\qquad
\text{for } t\to 0^+ \ ,
\]
where $c>0$ and $a=2\mHC+1$, and the propagator (semi-)norm pertaining to the evolution of~\eqref{DAE:EA} reads
\[
\|S(t)\|_{\mE}
:=
\sup_{\stackrel{\|x(0)\|_{\mE}=1}{\text{for consistent } x(0)}}
\|x(t)\|_{\mE} \ ,
\qquad
t\geq 0 \ .
\]
\end{proposition}

The anisotropic Oseen model~\eqref{model:Oseen:anisotropic} with constant $b\in\R^2$ and $b_1\ne 0$ is not hypocoercive, see Proposition~\ref{prop:Oseen:constantU:notHC}. 
This lack of hypocoercivity can also be verified by considering \eqref{model:Oseen:anisotropic} with $b\in\R^2$ and $b_1\ne 0$ as a partial differential-algebraic equation (PDAE), and bringing its modal representation into staircase form:
\begin{eexample}\label{ex:Oseen:constantU}
Consider the anisotropic Oseen model~\eqref{model:Oseen:anisotropic} with constant vector $b\in\R^2$.
Proceeding as in~\S\ref{ssec:Stokes} yields with $\mP_k$ from~\eqref{ex:Stokes:Pk}:
\[
 \widecheck \mE\ \dot y_k (t) 
 =(\widecheck \mJ_k -\widecheck \mR_k)\ y_k(t) 
 \ , \qquad %
 t\geq 0 \ ,
\]
with $\widecheck\mE =\diag(1,1,0)$, $\widecheck\mR_k =\diag(\nu k_2^2,\nu k_2^2, 0)$, and $\widecheck \mJ_k$ as in~\eqref{checkJk}.
The modes with $k_2=0$ (and $k_1\ne 0$) imply $y_{k,1}(t)= \ukI (t)\,{k_1}/{|k_1|} =0$, which is also a consistency condition on the initial value~$\ukI(0)$, i.e.\ the divergence-free condition of these initial modes. 
Since the corresponding $\widecheck\mR_k=0$, the modes with $k=(k_1,0)$ are not hypocoercive. 
In fact, they are purely oscillatory and have no damping.

This modal approach also shows that general solutions to~\eqref{model:Oseen:anisotropic} in~$\wcH$ converge with rate~$\mu=\nu$ to traveling waves like~\eqref{counterex:Oseen:constantU:notHC}.
\end{eexample}

\subsection{Review of Algorithm 3 from \cite{AAM21}}
\label{app:algorithm:1Pi}

The purpose of Algorithm 3 from \cite{AAM21} is to construct an ansatz for strict Lyapunov functionals for semi-dissipative Hamiltonian ODEs~\eqref{ODE:A} with negative hypocoercive matrix $\mA\in\Cnn$.
In~\cite{AAC22b},  explicit restrictions on~$\epsilon_j$ (relative to other parameters) were derived such that a suitable choice of $\epsilon_j$ turns the ansatz in Step~10 of Algorithm~\ref{algorithm:1Pi} into a strict Lyapunov functional.

Consider a semi-dissipative matrix~$\mA=\mJ-\mR$ with finite HC-index, then Algorithm~3 reads as follows:

\smallskip
\setcounter{algorithm}{2} 
\begin{breakablealgorithm}
\caption{Construction of a strict Lyapunov functional}
\label{algorithm:1Pi}
\begin{algorithmic}[1]
\Require $\Pi_0:=\mI$, $\tA_0:=-\mJ$, $\tB_0:=\mR$, $j:=1$
\State Construct an orthogonal projection~$\widetilde{\Pi}_j$ onto $\ker \big(\tB_{j-1} \tB_{j-1}^H\big)$.
\State $\Pi_j :=\widetilde{\Pi}_j \Pi_{j-1}$
\While{$\Pi_j \ne0$}
\State Set $\tA_j :=\Pi_j \tA_{j-1} \Pi_j$, $\tB_j :=\Pi_j \tA_{j-1} (\Pi_{j-1} -\Pi_j)$.
\State $j := j+1$
\State Construct an orthogonal projection~$\widetilde{\Pi}_j$ onto $\ker \big(\tB_{j-1} \tB_{j-1}^H\big)$.
\State $\Pi_j :=\widetilde{\Pi}_j \Pi_{j-1}$
\EndWhile
\State $\mHC :=j-1$
\State Set $\mX :=\Pi_0 +\sum_{j=1}^{\mHC} \epsilon_j \big(\tA_{j-1}\Pi_j +\Pi_j \tA_{j-1}^H\big)$ for sufficiently small $\epsilon_j>0$.
\Ensure $\|\cdot\|_\mX^2 :=\ip{\cdot}{\mX\cdot}$
\end{algorithmic}
\end{breakablealgorithm}
%

\section{Auxiliary results for the infinite dimensional ODE (\ref{ODE:Oseen:y2})}
\subsection{Quantitative estimate of~$\kappa$ for the evolution generators from Proposition~\ref{prop:Oseen:y2}} 
\label{app:QuantitativeEstimate}

\newcommand{\hj}{\widehat{j}}
\newcommand{\hk}{\widehat{k}}
\newcommand{\hkm}{\hk_-}
\newcommand{\hkp}{\hk_+}
\newcommand{\hl}{\widehat{\ell}}
\newcommand{\normalized}[1]{\frac{#1}{|#1|}}
Here we compute the HC-index of the system matrix~$(\hR-\hJ_{k_1})_{2,2}$ in the modal dynamics~\eqref{ODE:Oseen:y2} using directly Definition~\ref{def:unb-op:HC-index}, and derive a quantitative estimate for $\kappa>0$ in~\eqref{Op:J-R:unbdd}.
This is an alternative approach to the proof given for Proposition~\ref{prop:Oseen:y2}.
\begin{proposition}\label{prop:Oseen:y2:kappa}
Let $b_1=\sin (x_2)$ in~\eqref{model:Oseen:anisotropic_2}.
Then, for each $k_1\in\Zo$, the system matrix~$-{\mA}_{k_1}:=(\hR-\hJ_{k_1})_{2,2}$ in the modal dynamics~\eqref{ODE:Oseen:y2} has HC-index $\mHC=1$, and satisfies~\eqref{Op:J-R:unbdd} for some $\kappa\geq \nu/100>0$.
\end{proposition}
\begin{proof}
We show that there exists $\kappa>0$ such that~\eqref{Op:J-R:unbdd} with $m=1$, $\mJ=\hJzz:=\hJkzz$ and $\mR=\hRzz$ holds.
Let $x\in\cD(\VhRzz)$ be a unit vector, and decompose $x=v+w$ where $v\in(\ker(\VhRzz))^\perp$ and $w\in\ker(\VhRzz)$.
Define $\alpha:=\|v\|_2$, and then supposing $0<\alpha<1$, define the unit vector $u:=v/\alpha$ and note that $w/\sqrt{1-\alpha^2}=e_0$, cf.~\eqref{DAE:Oseen:staircase:R}.
We now compute 
\begin{equation} \label{est:xRx}
 \normBig{\VhRzz x}^2
=\normBig{\VhRzz v}^2
=\alpha^2 \normBig{\VhRzz u}^2
\geq \alpha^2 \nu
\end{equation}
using~\eqref{DAE:Oseen:staircase:R}.
In the same way, we find 
\begin{equation} \label{est:xJRJx}
\begin{split}
&\normBig{\VhRzz\hJzz^* x}^2 \\
&=\ipBig{\VhRzz\hJzz^* (v+w)}{\VhRzz\hJzz^* (v+w)} \\
&=\normBig{\VhRzz\hJzz^* v}^2 +2\Re\Big(\ipBig{\VhRzz\hJzz^* v}{\VhRzz\hJzz^* w}\Big) +\normBig{\VhRzz\hJzz^* w}^2 \\
&=\alpha^2 \normBig{\VhRzz\hJzz^* u}^2 +\alpha\sqrt{1-\alpha^2} 2\Re\Big(\ipBig{\VhRzz\hJzz^* u}{\VhRzz\hJzz^* e_0} \Big) 
\\
&\qquad +(1-\alpha^2) \normBig{\VhRzz\hJzz^* e_0}^2 \ .
\end{split}
\end{equation}

\medskip\noindent
\underline{Step 1:} We now estimate from below the three terms of~\eqref{est:xJRJx}.
For $k_1\ne 0$, the operator $\hJzz:=(\hJ_{k_1})_{2,2}$ given in~\eqref{DAE:Oseen:staircase:J} can be written as 
\begin{multline} \label{hJzz}
 \hJzz (k_2,\ell) 
= \ip{e_{k_2}}{\hJzz e_{\ell}}
=\frac{k_1}{2} 
 \begin{cases}
  \pm \frac{\hk}{|\hk|}\cdot\frac{\hl}{|\hl|}\ , &\text{for } \ell=k_2\pm 1\ , \\
  0\ , &\text{else,}
 \end{cases}
\\
\text{where }
 \hk:=\begin{bmatrix} k_1 \\ k_2 \end{bmatrix}, \quad
 \hl:=\begin{bmatrix} k_1 \\ \ell \end{bmatrix} ,
\end{multline}
such that $\hJzz =-\hJzz^*$ and
\begin{equation*} 
 \Big((\hRzz)^{1/2} \hJzz^*\Big) (k_2,\ell) 
= \ip{e_{k_2}}{(\hRzz)^{1/2} \hJzz^* e_{\ell}}
=\sqrt{\nu}\ \frac{k_1}{2}
 \begin{cases}
  \pm |k_2| \frac{\hk}{|\hk|}\cdot\frac{\hl}{|\hl|}\ , &\text{for } \ell=k_2\mp 1 , \\
  0\ , &\text{else.}
 \end{cases}
\end{equation*}
In particular, we deduce
\begin{equation}
\begin{split}
 \Big(\VhRzz\hJzz^*\Big) e_0 (k_2) 
=\sqrt{\nu}\ \frac{k_1}{2}
 \begin{cases}
  \pm \frac{k_1^2}{\sqrt{k_1^2+1}\ |k_1|}\ , &\text{for } k_2=\pm 1\ , \\
  0\ , &\text{else,}
 \end{cases}
\end{split}
\end{equation}
such that
\begin{subequations} \label{est:xJRJx_parts}
\begin{equation}
 \normBig{\VhRzz\hJzz^* e_0}^2 
=\frac{\nu}2 \frac{k_1^4}{(k_1^2+1)} \ .
\end{equation}

\medskip\noindent
Continuing with the second term of~\eqref{est:xJRJx}, and with $k_1\ne 0$, we derive using $u_0=0$ that
\begin{equation}
\begin{split}
&2\Re\Big(\ipBig{\VhRzz\hJzz^* u}{\VhRzz\hJzz^* e_0}\Big)
 \\
&=-\nu \frac{k_1^2}{4} \frac{(k_1^2+2)\ k_1^2}{(k_1^2+1)\ |k_1|\ \sqrt{k_1^2+4}} 2 \big(\Re(u_{-2}) +\Re(u_2) \big)
 \\
&\geq -\nu \frac{k_1^2}{4} \frac{(k_1^2+2)\ k_1^2}{(k_1^2+1)\ |k_1|\ \sqrt{k_1^2+4}} 2 \big(|\Re(u_{-2})| +|\Re(u_2)| \big) \ .
\end{split}
\end{equation}

\medskip\noindent
For the first term of~\eqref{est:xJRJx}, and with $k_1\ne 0$, we compute using $u_0=0$ that
\begin{equation} \label{est:uJRJu}
\begin{split}
\big\|\VhRzz\hJzz^* u \big\|^2 
&=\nu \frac{k_1^2}4 \sum_{k_2=-\infty}^{\infty} \Big| -|k_2| \normalized{\hk}\cdot\normalized{\hkm} u_{k_2-1} +|k_2| \normalized{\hk}\cdot\normalized{\hkp} u_{k_2+1} \Big|^2 \\
&=\nu \frac{k_1^2}4 \sum_{k_2=-\infty}^{\infty} |k_2|^2\ \Big| -\normalized{\hk}\cdot\normalized{\hkm} u_{k_2-1} +\normalized{\hk}\cdot\normalized{\hkp} u_{k_2+1} \Big|^2 \\
&\geq \nu \frac{k_1^2}4 \sum_{k_2=-1,1} |k_2|^2\ \Big| -\normalized{\hk}\cdot\normalized{\hkm} u_{k_2-1} +\normalized{\hk}\cdot\normalized{\hkp} u_{k_2+1} \Big|^2 \\
 %
 %
&= \nu \frac{k_1^2}4  \frac{(k_1^2+2)^2}{(k_1^2+1)\ (k_1^2 +4)} \Big( |u_{-2}|^2 + |u_{2}|^2 \Big) \\
&\geq \nu \frac{k_1^2}4  \frac{(k_1^2+2)^2}{(k_1^2+1)\ (k_1^2 +4)} \Big( (\Re(u_{-2}))^2 + (\Re(u_{2}))^2 \Big)  
\end{split}
\end{equation}
where 
\[
 \hk =\begin{bmatrix} k_1 \\ k_2 \end{bmatrix}, \qquad
 \hkm :=\begin{bmatrix} k_1 \\ k_2-1 \end{bmatrix}, \qquad
 \hkp :=\begin{bmatrix} k_1 \\ k_2+1 \end{bmatrix}. \qquad
\]
\end{subequations}

\medskip\noindent
Combining~\eqref{est:xJRJx} and~\eqref{est:xJRJx_parts} yields
\begin{equation}\label{est:xJRJx_final}
\begin{split}
&\normBig{\VhRzz\hJzz^* x}^2
 \\
&\geq \alpha^2 \nu \frac{k_1^2}4  \frac{(k_1^2+2)^2}{(k_1^2+1)\ (k_1^2 +4)} \Big( (\Re(u_{-2}))^2 + (\Re(u_{2}))^2 \Big) \\
& \quad 
-\alpha\sqrt{1-\alpha^2} \nu \frac{k_1^2}{2} \frac{(k_1^2+2)\ k_1^2}{(k_1^2+1)\ |k_1|\ \sqrt{k_1^2+4}} \big(|\Re(u_{-2})| +|\Re(u_2)| \big) 
+(1-\alpha^2) \frac{\nu}2 \frac{k_1^4}{(k_1^2+1)}
 \\
&= \nu \frac{k_1^2}4 \Bigg( \alpha \frac{k_1^2+2}{\sqrt{k_1^2+1}\ \sqrt{k_1^2 +4}} |\Re(u_{-2})|  
-\sqrt{1-\alpha^2} \frac{|k_1|}{\sqrt{k_1^2+1}} \Bigg)^2 \\
& \quad
 +\nu \frac{k_1^2}4 \Bigg( \alpha \frac{k_1^2+2}{\sqrt{k_1^2+1}\ \sqrt{k_1^2 +4}} |\Re(u_2)|  
-\sqrt{1-\alpha^2} \frac{|k_1|}{\sqrt{k_1^2+1}} \Bigg)^2 \\
&\geq \frac{\nu}4 \Big( \alpha \ a |\Re(u_{-2})| -\sqrt{1-\alpha^2} \ b \Big)^2 
 +\frac{\nu}4 \Big( \alpha \ a |\Re(u_2)| -\sqrt{1-\alpha^2} \ b \Big)^2 \ ,
\end{split}
\end{equation}
where 
\begin{equation} \label{a:b}
 a:=\frac{k_1^2+2}{\sqrt{k_1^2+1}\ \sqrt{k_1^2 +4}}, \qquad
 b:=\frac{|k_1|}{\sqrt{k_1^2+1}}.
\end{equation}
Together with~\eqref{est:xRx}, for any $x\in\cD(\VhRzz)$ with $\norm{x}=1$, we obtain
\begin{multline} \label{est:kappa:0}
 \sum_{j=0}^1 \normBig{\VhRzz(\hJzz^*)^j x}^2
 \\
\geq \nu \alpha^2 +\frac{\nu}4 \Big( \alpha \ a |\Re(u_{-2})| -\sqrt{1-\alpha^2} \ b \Big)^2 
 +\frac{\nu}4 \Big( \alpha \ a |\Re(u_2)| -\sqrt{1-\alpha^2} \ b \Big)^2 \ .
\end{multline}

\medskip\noindent
\underline{Step 2:}
Finally, we show that there exists $\kappa>0$ such that~\eqref{Op:J-R:unbdd} holds with $m=1$, $\mJ=\hJzz:=\hJkzz$ and $\mR=\hRzz$.
To this end, we estimate the left-hand side of~\eqref{est:kappa:0} uniformly in $\alpha\in[0,1]$, $|\Re(u_2)|\in[0,1]$, and $k_1\in\N$:
We observe that the function
\begin{equation} \label{b}
 b^2:\
 [1,\infty)\to [0,\infty)\ , \quad 
 k_1 \mapsto \frac{k_1^2}{k_1^2+1}\ ,
\end{equation}
is monotonically increasing such that $b^2(k_1)\geq b^2(1) =1/2$ for $k_1\in[1,\infty)$.
Then, we estimate the expression on the left-hand side of~\eqref{est:kappa:0} as
\begin{equation} \label{est:kappa}
\begin{split}
& \sum_{j=0}^1 \normBig{\VhRzz(\hJzz^*)^j x}^2
 \\
&\geq \nu \alpha^2 +\frac{\nu}4 \Big( \alpha \ a |\Re(u_{-2})| -\sqrt{1-\alpha^2} \ b \Big)^2 
 +\frac{\nu}4 \Big( \alpha \ a |\Re(u_2)| -\sqrt{1-\alpha^2} \ b \Big)^2 
 \\
&= \nu \Big(\alpha^2 
 +\frac{b^2}4 \Big( \alpha \frac{a}{b}\ |\Re(u_{-2})| -\sqrt{1-\alpha^2} \Big)^2 
 +\frac{b^2}4 \Big( \alpha \frac{a}{b}\ |\Re(u_2)| -\sqrt{1-\alpha^2} \Big)^2 \Big) 
 \\
&\geq \nu \Big(\alpha^2 
 +\frac18 \Big( \alpha \frac{a}{b}\ |\Re(u_{-2})| -\sqrt{1-\alpha^2} \Big)^2 
 +\frac18 \Big( \alpha \frac{a}{b}\ |\Re(u_2)| -\sqrt{1-\alpha^2} \Big)^2 \Big) 
 \\
&\geq \nu \Big(\alpha^2 
 +\frac18 \Big( \alpha \frac{a}{b}\ |\Re(u_2)| -\sqrt{1-\alpha^2} \Big)^2 \Big)\ .
\end{split}
\end{equation}
We continue to derive a lower bound for the function
\begin{equation} \label{h:alpha:beta}
 h:\
 [0,1]\times [0,1] \to [0,\infty)\ , \qquad
 (\alpha,\beta) \mapsto \alpha^2 
 +\frac18 \Big( \alpha \beta c -\sqrt{1-\alpha^2} \Big)^2 \ ,
\end{equation}
where $c:=a/b>0$ and $\beta$ replaces $|\Re(u_2)|$.
The function
\begin{equation} \label{c2}
 c^2:\
 [1,\infty)\to [0,\infty)\ , \quad 
 k_1 \mapsto \frac{a^2}{b^2} =\frac{(k_1^2+2)^2}{k_1^2\ (k_1^2 +4)}\ ,
\end{equation}
is monotonically decreasing such that $1\leq c^2(k_1)\leq c^2(1) =9/5$ for $k_1\in[1,\infty)$.

A straightforward computation shows that~$h$ has no extremum in the interior of~$[0,1]^2$.
Therefore, the minimum of~$h=h(\alpha,\beta)$ in~\eqref{h:alpha:beta} is located at the boundary of $(\alpha,\beta)\in[0,1]\times[0,1]$:
For $\beta\in[0,1]$, we derive
\[
 h(0,\beta) =\frac18 \ , \qquad
 h(1,\beta) =1 +\frac18 c^2 \beta^2 \geq 1 \ .
\]
For $\alpha\in[0,1]$, we derive
\[
 h(\alpha,0)
=\frac18 (7\alpha^2 +1)
\geq \frac18 \ ,
\]
and
\begin{equation} \label{h:alpha:1}
 h(\alpha,1)
=\alpha^2 +\frac18 \Big( \alpha c -\sqrt{1-\alpha^2} \Big)^2 \ .
\end{equation}
To finish the estimate, we have to derive a lower bound for $h(\alpha,1)$, $\alpha\in[0,1]$:
We observe that $h(0,1)=1/8$ and $h(1,1)=1 +c^2/8\geq 1$.
To find local extrema, we search for $\alpha_*\in(0,1)$ such that 
\[
 0
=\frac{\partial h}{\partial\alpha} (\alpha_*,1)
=\underbracket{\frac{\alpha_*}4 (7 +c^2)}_{>0}
 +\underbracket{\frac{c}4 \frac{2\alpha_*^2 -1}{\sqrt{1-\alpha_*^2}}}_{<0 \text{ on } (0,1/\sqrt{2})} \ ,  
\]
or,
\[
 \alpha_*^4 -\alpha_*^2 +\frac{c^2}{4c^2 +(7+c^2)^2} =0 \ .
\]
Solving for $\alpha_*^2$, we find the solutions 
\[
 (\alpha_*^2)_\pm 
=\frac12 \Big(1\pm \sqrt{\frac{(7+c^2)^2}{4c^2 +(7+c^2)^2}}\Big)
\in (0,1) \ .
\]
Using the positive root of $(\alpha_*^2)_-$ yields
\[
 h\Big(\sqrt{ (\alpha_*^2)_- },1\Big) 
\geq (\alpha_*^2)_- 
=\frac12 \Big(1 -\sqrt{\frac{(7+c^2)^2}{4c^2 +(7+c^2)^2}}\Big) >0 \ .
\]

\medskip
To derive a lower bound on~$h$, which is uniform w.r.t. $k_1\in[1,\infty)$, we recall that $1\leq c^2(k_1)\leq 2$ for $k_1\in[1,\infty)$ and study the function
\[
 g:\
 [1,2]\to [0,\infty)\ , \qquad
 \gamma \mapsto \frac{(7+\gamma)^2}{4\gamma +(7+\gamma)^2}\ .
\]
The function $g=g(\gamma)$ is monotone decreasing w.r.t. $\gamma\in[1,2]$ such that $81/89=g(2) \leq g(\gamma) \leq g(1)=64/68$.
This implies that 
\[
 (\alpha_*^2)_- =\frac12 \Big(1 -\sqrt{g(c^2)}\Big)
\geq \frac12 \Big(1 -\sqrt{g(1)}\Big) =0.0149\ldots >1/100 \ .
\]
Altogether, we derive that 
\[ 
 h(\alpha,\beta)>1/100 \qquad\text{for } (\alpha,\beta)\in[0,1]\times[0,1] \ ,
\]
which implies the uniform lower bound $\kappa\geq \nu/100$ in~\eqref{est:kappa}.
\end{proof}


\subsection{Proof of Lemma \ref{lemma:Oseen:LMI}}\label{appendix_three}
\begin{proof}
To prove uniform coercivity (w.r.t.\ $k_1\in\Z\setminus\{0\}$) of the self-adjoint operator $\mQ_{k_1} :=-(\mA_{k_1}^H \mX_{k_1} +\mX_{k_1}\mA_{k_1})$, we use $\mX_{k_1} =\mI +\epsilon_{k_1}\mY$ in~\eqref{model:Oseen:X} and $\mA_{k_1} =(\widehat{\mJ}_{k_1})_{2,2} -\widehat{\mR}_{2,2}$ to deduce that 
\[ \mQ_{k_1} =2\hRzz -\epsilon_{k_1} (\mA_{k_1}^H \mY +\mY\mA_{k_1}) \ . \]
Hence, $\mQ_{k_1}$ is the sum of the diagonal operator~$2\hRzz$, and the compact operator~$-\epsilon_{k_1} (\mA_{k_1}^H \mY +\mY\mA_{k_1})$ acting on the finite-dimensional subspace $\cH_5:=\spn\{e_{-2},e_{-1},e_0,e_1,e_2\}$.
Since $\ip{x}{\mQ_{k_1} x}=\ip{x}{2\hRzz x}\geq 18\|x\|^2$ for all $x\in(\cH_5)^\perp$, we are left to prove uniform coercivity of $\mQ_{k_1}\big|_{\cH_5}$, or equivalently, of the following $5\times 5$-matrix:
For $\alpha :=\epsilon_{k_1} k_1$, the matrices~$\mQ$ representing $\mQ_{k_1}\big|_{\cH_5}$ read
\[
 \mQ
 =
 \begin{bmatrix}
  8\nu & 0 & -\alpha\beta_2 & 0 & 0 \\
  0 & -2\alpha\beta_1 +2\nu & -\alpha\nu/k_1 & 2\alpha\beta_1 & 0 \\
  -\alpha\beta_2 & -\alpha\nu/k_1 & 4\alpha\beta_1 & \alpha\nu/k_1 & -\alpha\beta_2 \\
  0 & 2\alpha\beta_1 & \alpha\nu/k_1 & -2\alpha\beta_1 +2\nu & 0 \\
  0 & 0 & -\alpha\beta_2 & 0 & 8\nu
 \end{bmatrix}
\]
where $\beta_1, \beta_2$ are functions given as
\[
 \beta_1: 
 \Z\to\R\ , \quad 
 k_1 \mapsto \frac{k_1^2}{2\sqrt{k_1^2 +1}\ |k_1|} \ ,
 \qquad \text{and }
 \beta_2: 
 \Z\to\R\ , \quad 
 k_1 \mapsto \frac{k_1^2+2}{2\sqrt{k_1^2 +4}\ \sqrt{k_1^2 +1}} \ ,
\]
which satisfy
\begin{equation}\label{bounds:beta1} 
\beta_{1,\min} := \frac1{2\sqrt{2}} \leq \beta_1(k_1) < \frac12 
\quad 
\text{for all } k_1\in\Zo \ ,
\end{equation}
and
\begin{equation}\label{bounds:beta2}
\beta_{2,\min} := \frac{3}{\sqrt{40}} \leq \beta_2(k_1) \leq \frac12 
\quad 
\text{for all } k_1\in\Zo \ .
\end{equation}
A Hermitian matrix is positive definite if and only if all of its leading principal minors are positive definite \cite{HoJo13}.
Using permutations of rows and columns, it is evident that all principal minors have to be positive definite.
Indeed, we consider other (not only the leading) principal minors to highlight restrictions on $\alpha$.

The $1\times 1$ minors are the diagonal elements of~$\mQ$.
The leading principal $1\times 1$ minor $\mQ_{1,1}=8\nu>0$ is positive, since the diffusion coefficient $\nu>0$ is positive.
The coefficient $\mQ_{2,2} =-2\alpha\beta_1 +2\nu$ is positive if and only if $\alpha \beta_1 < \nu$ for all $k_1\in\Zo$ which holds if 
\begin{equation}\label{minor:Q2}
 \alpha \le 2\nu 
 \quad 
 \text{for all $k_1\in\Zo$} ,
\end{equation}
due to~\eqref{bounds:beta1}.
The coefficient $\mQ_{3,3} =4\alpha\beta_1$ is positive if and only if 
$0<\alpha \beta_1$ for all $k_1\in\Zo$, which holds if
\begin{equation}\label{minor:Q3}
 0<\alpha \quad \text{for all $k_1\in\Zo$} ,
\end{equation}
due to~\eqref{bounds:beta1}.

The leading principal $2\times 2$ minor $\det\mQ_{\{1,2\}\times\{1,2\}}=8\nu\ (-2\alpha\beta_1 +2\nu)$ is positive, since it is the product of two diagonal elements.
The principal minor $\det\mQ_{\{2,4\}\times\{2,4\}}=(-2\alpha\beta_1 +2\nu)^2 -(2\alpha\beta_1)^2 =2\nu\ (-4\alpha\beta_1 +2\nu)>0$, if and only if $-4\alpha\beta_1 +2\nu>0$ for all $k_1\in\Zo$, which is satisfied if 
\begin{equation}\label{minor:Q24}
 \alpha \le\nu 
 \quad 
 \text{for all $k_1\in\Zo$} ,
\end{equation}
due to~\eqref{bounds:beta1}.
The principal minor $\det\mQ_{\{2,3\}\times\{2,3\}}=\alpha\big(8\beta_1 \nu -\alpha(8\beta_1^2 +\nu^2/k_1^2)\big)$ is positive if and only if $\alpha< 8\beta_1 \nu /(8\beta_1^2 +\nu^2/k_1^2)$ for all $k_1\in\Zo$.
This holds if 
\begin{equation}\label{minor:Q23}
 \alpha <\nu \frac{2\sqrt{2}}{2+\nu^2}
 \quad 
 \text{for all $k_1\in\Zo$} ,
\end{equation}
since ${2\sqrt{2}}/(2+\nu^2) \leq 8\beta_1 /(8\beta_1^2 +\nu^2/k_1^2)$ for all $k_1\in\Zo$ due to~\eqref{bounds:beta1}.

The leading principal $3\times 3$ minor $\det\mQ_{\{1,2,3\}\times\{1,2,3\}}=2\beta_1 \beta_2^2 \alpha^3 -2(32\beta_1^2 \nu +\beta_2^2 \nu +4\nu^3 /{k_1^2})\ \alpha^2 +64\beta_1 \nu^2 \alpha$ is positive if  
\[
 \underbracket{\beta_1 \beta_2^2}_{=:a_3} \alpha^2 \underbracket{-\nu (32\beta_1^2 +\beta_2^2 +4\nu^2)}_{=:b_3} \alpha +\underbracket{32\beta_1 \nu^2}_{=:c_3} 
 >
 0
\]
for all $k_1\in\Zo$.
This quadratic polynomial in~$\alpha$ has two positive roots, since $b_3<0$ and $c_3>0$.
Moreover, since $4 a_3 c_3$ is uniformly bounded away from 0, there exists a constant $\alpha_{3,\min}>0$,  independent of $k_1$, which is a lower bound (uniformly in $k_1\in\Zo$) for the smaller root~$\alpha^{(3)}_{-} :=(-b_3 -\sqrt{b_3^2 -4 a_3 c_3})/(2a_3)$.
Consequently, for $\alpha\in(0,\alpha_{3,\min})$, the leading principal $3\times 3$ minor $\det\mQ_{\{1,2,3\}\times\{1,2,3\}}$ is positive for all $k_1\in\Zo$. 

The leading principal $4\times 4$ minor $\det\mQ_{\{1,2,3,4\}\times\{1,2,3,4\}}=8\beta_1 \beta_2^2 \nu \alpha^3 -4(64\beta_1^2 \nu^2 +\beta_2^2 \nu^2 +8\nu^4 /{k_1^2})\ \alpha^2 +128\beta_1 \nu^3 \alpha$ is positive if  
\[
 \underbracket{2\beta_1 \beta_2^2}_{=:a_4} \alpha^2 \underbracket{-\nu (64\beta_1^2 +\beta_2^2 +8\nu^2)}_{=:b_4} \alpha +\underbracket{32\beta_1 \nu^2}_{=:c_4} 
 >
 0
\]
for all $k_1\in\Zo$.
This quadratic polynomial in~$\alpha$ has two positive roots, since $b_4<0$ and $c_4>0$.
Moreover, since $4 a_4 c_4$ is uniformly bounded away from 0, there exists a constant $\alpha_{4,\min}>0$, independent of $k_1$, which is a lower bound (uniformly in $k_1\in\Zo$) for the smaller root~$\alpha^{(4)}_{-} :=(-b_4 -\sqrt{b_4^2 -4 a_4 c_4})/(2a_4)$. 
Consequently, for $\alpha\in(0,\alpha_{4,\min})$, the leading principal $4\times 4$ minor $\det\mQ_{\{1,2,3,4\}\times\{1,2,3,4\}}$ is positive for all $k_1\in\Zo$.

The (leading) principal $5\times 5$ minor $\det\mQ=128\beta_1 \beta_2^2 \nu^2 \alpha^3 -64(32\beta_1^2 \nu^3 +\beta_2^2 \nu^3 +4\nu^5 /{k_1^2})\ \alpha^2 +1024\beta_1 \nu^4 \alpha$ is positive if  
\[
 \underbracket{2\beta_1 \beta_2^2}_{=:a_5} \alpha^2 \underbracket{-\nu (32\beta_1^2 +\beta_2^2 +4\nu^2)}_{=:b_5} \alpha +\underbracket{16\beta_1 \nu^2}_{=:c_5} 
 >
 0
\]
for all $k_1\in\Zo$.
This quadratic polynomial in~$\alpha$ has two positive roots, since $b_5<0$ and $c_5>0$.
Moreover, since $4 a_5 c_5$ is uniformly bounded away from 0, there exists a constant $\alpha_{5,\min}>0$, independent of $k_1$, which is a lower bound (uniformly in $k_1\in\Zo$) for the smaller root~$\alpha^{(5)}_{-} :=(-b_5 -\sqrt{b_5^2 -4 a_5 c_5})/(2a_5)$.
Consequently, for $\alpha\in(0,\alpha_{5,\min})$, the leading principal $5\times 5$ minor $\det\mQ$ is positive for all $k_1\in\Zo$.

Thus, we choose
\begin{equation}\label{cond:Oseen:alpha}
 0 <\alpha <\min\big\{ 1/{\sqrt2},\ \nu,\ \nu\frac{2\sqrt{2}}{2+\nu^2},\, \alpha_{3,\min},\ \alpha_{4,\min},\ \alpha_{5,\min} \big\} =: \alpha_{\min},
\end{equation}
due to the restriction $|\epsilon_{k_1}|<1/{\sqrt2}$ to ensure that $\mX_{k_1}$ is positive definite and $\epsilon_{k_1}= \alpha/k_1$, and those conditions to ensure that $\mQ =\mQ_{k_1}$ is positive definite.\\

Finally we  determine $\mu_{k_1}>0$ (bounded below, uniformly in $k_1\in\Zo$) such that the LMIs~\eqref{DAE:Oseen:LMI} hold. 
Let $\{\lambda_1,\ \lambda_2,\ \ldots,\ \lambda_5\}$ be the eigenvalues of the positive definite Hermitian matrix~$\mQ$ arranged in increasing order.
We seek a lower bound on $\lambda_1$.
Note that the arithmetic-geometric mean inequality yields
\[
 \lambda_1
 =
 \frac{\det\mQ}{\lambda_2 \lambda_3 \lambda_4 \lambda_5}
 \geq 
 \big(\frac{\lambda_2 +\lambda_3 +\lambda_4 +\lambda_5}4\big)^{-4} \det{\mQ} 
 \geq 
 256 \frac{\det\mQ}{(\tr\mQ)^4} .
\]
Since $\tr\mQ =20\nu$ is independent of $k_1$, we finally obtain the bound
\begin{equation}\label{model:Oseen:lambda1}
\begin{split}
 \lambda_1
&\geq 
 256 \frac{\det\mQ}{(\tr\mQ)^4}
\\
&=
 \big( 128\beta_1 \beta_2^2 \nu^2 \alpha^3 -64(32\beta_1^2 \nu^3 +\beta_2^2 \nu^3 +4\nu^5 /{k_1^2})\ \alpha^2 +1024\beta_1 \nu^4 \alpha\big)  \frac{256}{(20\nu)^4}
\\
&= 
 \big( 2\beta_1 \beta_2^2 \alpha^2 -(32\beta_1^2 +\beta_2^2 +4\nu^2 /{k_1^2})\nu \alpha +16\beta_1 \nu^2 \big) 64\nu^2 \alpha \frac{256}{(20\nu)^4} 
\\
&>  
 \big( 2\beta_{1,\min} \beta_{2,\min}^2 \alpha^2 -({33}/{4} +4\nu^2)\nu \alpha +16\beta_{1,\min} \nu^2\big) \alpha \frac{64}{5^4 \nu^2} 
 =:\lambda_{1,\min}(\alpha) \ .
\end{split}
\end{equation}
Note that $\lambda_{1,\min}>0$ for $\alpha>0$ small enough, and it is a lower bound on $\lambda_1$, uniform in $k_1\in\Zo$.
Then, for each admissible $\alpha$ from~\eqref{cond:Oseen:alpha} such that $\lambda_{1,\min}(\alpha)>0$, the uniform estimates $\mQ\geq \lambda_{1,\min}(\alpha)\mI$ and $\mX_{k_1} \leq (1+\sqrt{2}|\epsilon_{k_1}|)\ \mI \leq 2\mI$ using $|\epsilon_{k_1}|\leq 1/\sqrt{2}$ imply that $\mQ \geq \lambda_{1,\min}\mI \geq (\lambda_{1,\min}/2) \mX_{k_1}$.
Hence, the inequalities~\eqref{DAE:Oseen:LMI} hold for some constants $\mu_{k_1}>0$ which are uniformly (w.r.t.~$k_1$) bounded from below. 
One may choose, e.g., $\mu_{k_1} =\lambda_{1,\min}/4$, $k_1\ne0$.
\end{proof}

\section*{Declarations}

\textbf{Ethical Approval}
is not applicable.

\medskip\noindent
\textbf{Authors' contributions}
All authors wrote and reviewed the manuscript in a collaborative process.

\medskip\noindent
\textbf{Funding}
The first author (FA) was supported by the Austrian Science Fund (FWF) via the FWF-funded SFB \# F65.
The second author (AA) was supported by the Austrian Science Fund (FWF), partially via the FWF-doctoral school "Dissipation and dispersion in non-linear partial differential equations'' (\# W1245) and the FWF-funded SFB \# F65.
The third author (VM) was supported by Deutsche Forschungsgemeinschaft (DFG) via the DFG-funded SFB \# 910.

\medskip\noindent
\textbf{Availability of data and materials}
In this research no datasets were generated or analysed.

\medskip\noindent
\textbf{Conflict of interest}
The authors declare that they have no conflict of interest.

\medskip\noindent
\textbf{Acknowledgements}
For the purpose of open access, the authors have applied a CC BY public copyright licence to any Author Accepted Manuscript version arising from this submission.

\end{appendices}


\bibliography{AAM-2022-HCIndex}

\end{document}